\documentclass{amsart}
\usepackage[english]{babel}
\usepackage[utf8]{inputenc}
\usepackage{booktabs}
\usepackage{lmodern}
\usepackage[T1]{fontenc}
\usepackage{comment}
\usepackage{tikz}
	\usetikzlibrary{calc}
	\usetikzlibrary{decorations.pathmorphing}
	\usetikzlibrary{decorations.shapes}
	\tikzset{%
		cross/.style = {preaction={draw=white,line width=4pt}},
		dot1/.style = {draw,circle,minimum size=3pt,inner sep=0pt,outer sep=3pt,thick,fill},
		dot2/.style = {draw,circle,minimum size=3pt,inner sep=0pt,outer sep=3pt,thick},
		dot3/.style = {draw,rectangle,minimum size=3pt,inner sep=0pt,outer sep=3pt,thick,fill},
		dot4/.style = {draw,rectangle,minimum size=3pt,inner sep=0pt,outer sep=3pt,thick},
		steiner1/.style = {},
		steiner2/.style = {dash pattern=on .2cm off .1cm,thick},
		steiner3/.style = {very thick,dotted},
		steiner4/.style = {decorate,decoration={crosses,segment length=.18cm}},
		steiner5/.style = {decorate,decoration={zigzag,amplitude=.04cm}},
		steiner6/.style = {decorate,decoration={shape backgrounds,shape=rectangle,shape size=.05cm,shape sep=.1cm}},
	}
\usepackage{microtype}
\usepackage[subrefformat=parens,labelfont=rm]{subcaption}
\usepackage{enumitem}
	\setenumerate{label=(\roman*),leftmargin=*,widest*=3}
	\setitemize{leftmargin=*}
\usepackage{amsmath}
\usepackage{amssymb}
\usepackage{amsthm}
\theoremstyle{plain}
	\newtheorem{theorem}{Theorem}[section]
	\newtheorem{lemma}[theorem]{Lemma}
	\newtheorem{proposition}[theorem]{Proposition}
	\newtheorem{maintheorem}{Theorem}
	
\theoremstyle{definition}
	\newtheorem{definition}[theorem]{Definition}
\theoremstyle{remark}
	\newtheorem*{remark*}{Remark}
\usepackage{hyperref}
\DeclareMathOperator{\rank}{rank}
\graphicspath{{images/}}

\providecommand{\R}{\mathbb R}
\newcommand*\gsquare{\mathbin{\protect\scalebox{.707}{$\square$}}}
\newcommand*\srs{\textup{\textsf{srs}}}
\newcommand*\ths{\textup{\textsf{ths}}}

\begin{document}

%\title{Periodic Steiner graphs minimizing length} -- title of version 1
\title{Periodic Steiner networks minimizing length}
\date{\today}
\author[Alex \& Grosse-Brauckmann]{Jerome Alex, Karsten Grosse-Brauckmann}
\address{Technische Universit\"at Darmstadt, Fachbereich Mathematik (AG 3),
    Schlossgartenstr.~7, 64289 Darmstadt, Germany}
\email{jalex@mathematik.tu-darmstadt.de, kgb@mathematik.tu-darmstadt.de}

\begin{abstract}
\noindent
  We study a problem of geometric graph theory:
  We determine the triply periodic graph in Euclidean $3$-space 
  which minimizes length among all graphs 
  spanning a fundamental domain of $3$-space with the same volume.  
  The minimizer is the so-called \textsf{srs} network 
  with quotient the complete graph on four vertices~$K_4$.
  The network spans the body centred cubic lattice
  and is related to the gyroid triply periodic surface.
\end{abstract}

\maketitle

\section{Introduction}

Given a finite set of points, the Steiner problem is to find
a tree of minimal length connecting them~\cite{minimalnetworks}.
While this is a classical problem for the plane,
the case of dimension~$3$ and higher has received less attention.
Trees minimizing length usually have further vertices
which necessarily are of degree~$3$,
where the incident edges are coplanar and meet at $120^\circ$-angles. 
This is valid for any dimension, 
and we call this the \emph{Steiner condition}.

Here we consider infinite graphs without terminal vertices
in Euclidean space which are multiply periodic and have a finite quotient.
We call these graphs \emph{networks} or, 
if all vertices have degree~$3$ and the Steiner condition is met, 
\emph{Steiner networks}. 
We make these notions precise in Section~\ref{sec:stnetworks}.

We are interested in the case of dimension~$3$,
which has a strong motivation by surface theory.
There are various self-assembling biological and chemical systems 
which give rise to triply periodic interfaces.
As pointed out for instance in~\cite{deCampo},  
the most prevalent geometry is the \emph{gyroid}, 
a triply periodic embedded surface 
with the body-centred cubic lattice and quotient surface of genus~$3$.
Alan Schoen discovered the gyroid minimal surface in the 1970's 
in terms of a Steiner network~\cite{gyroidschoen} 
(see also~\cite{GrosseBrauckmann1996}).
By calling this network~\srs, a name which 
refers to the strontium silicide $\mathrm{SrSi_2}$ crystal,
we follow a crystallographic convention
(see~\cite{chemistrystructure} and also \href{http://www.rcsr.net}{rcsr.net}).
Other names for the network are Laves~\cite{coxeter}
or $(10,3)$-a~\cite{wells} (see also~\cite{schoennotices}).

The \textsf{srs} network, shown in Figure~\ref{fig:3D},
is highly symmetric, with symmetry group~$I4_132$.
Its quotient under the body-centred cubic lattice 
is the complete graph on $4$~vertices~$K_4$, 
see Figure~\ref{fig:abstractlist3}.
However, Steiner networks with $4$ vertices in the quotient
exist for arbitrary lattices (see Theorem~\ref{th:minimizer}).
Since nature neither assumes symmetries nor the choice of a particular
graph, this raises the question:
\emph{Is there a simple property distinguishing the $I4_132$-symmetric}
\textsf{srs} \emph{network from all other networks?}
\begin{figure}
	\centering
	\begin{subfigure}{.865\linewidth}
	\centering
	\includegraphics[width=\linewidth]{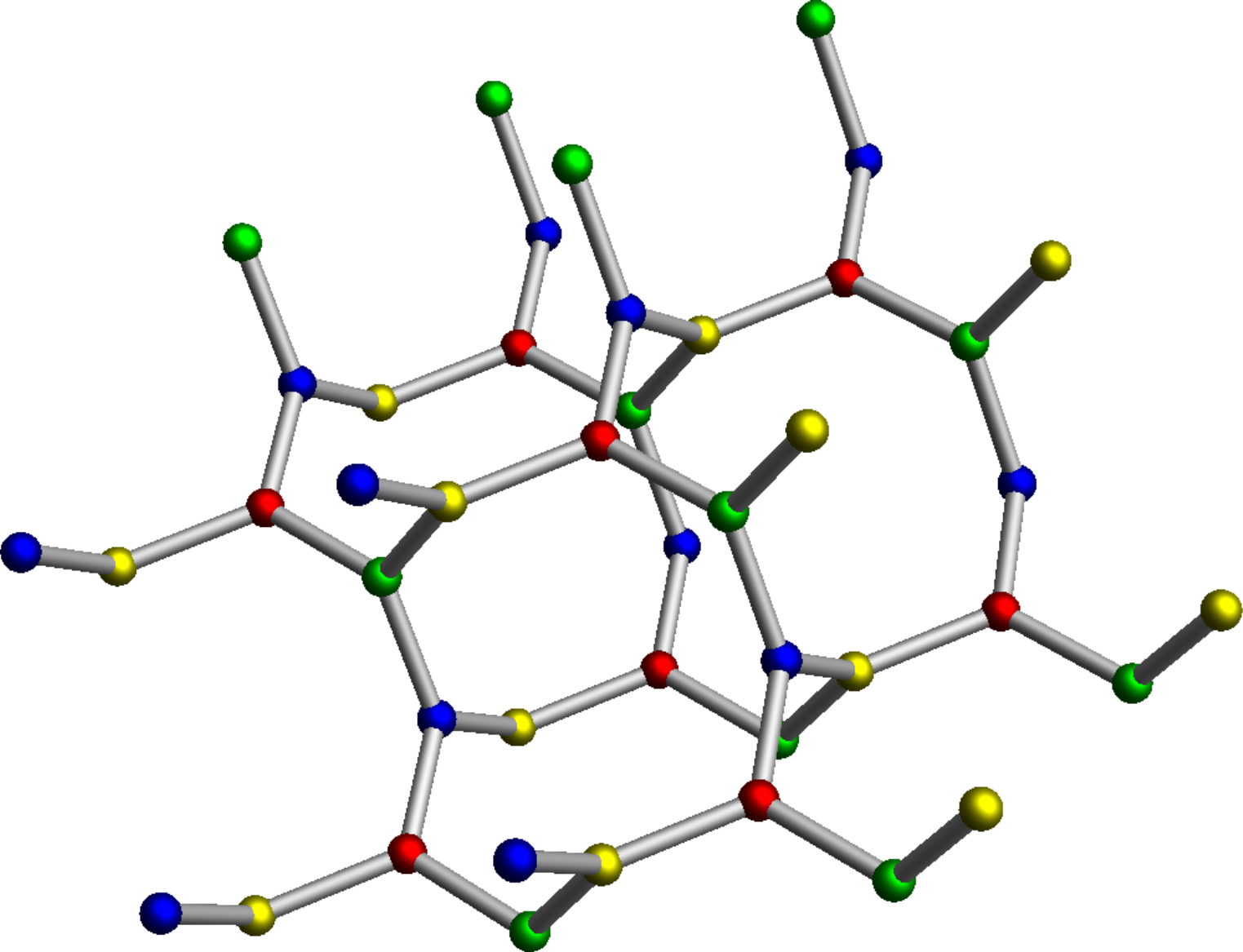}
	\end{subfigure}
	\vspace*{5mm}
	\begin{subfigure}{.78\linewidth}
	\centering
	\includegraphics[width=\linewidth]{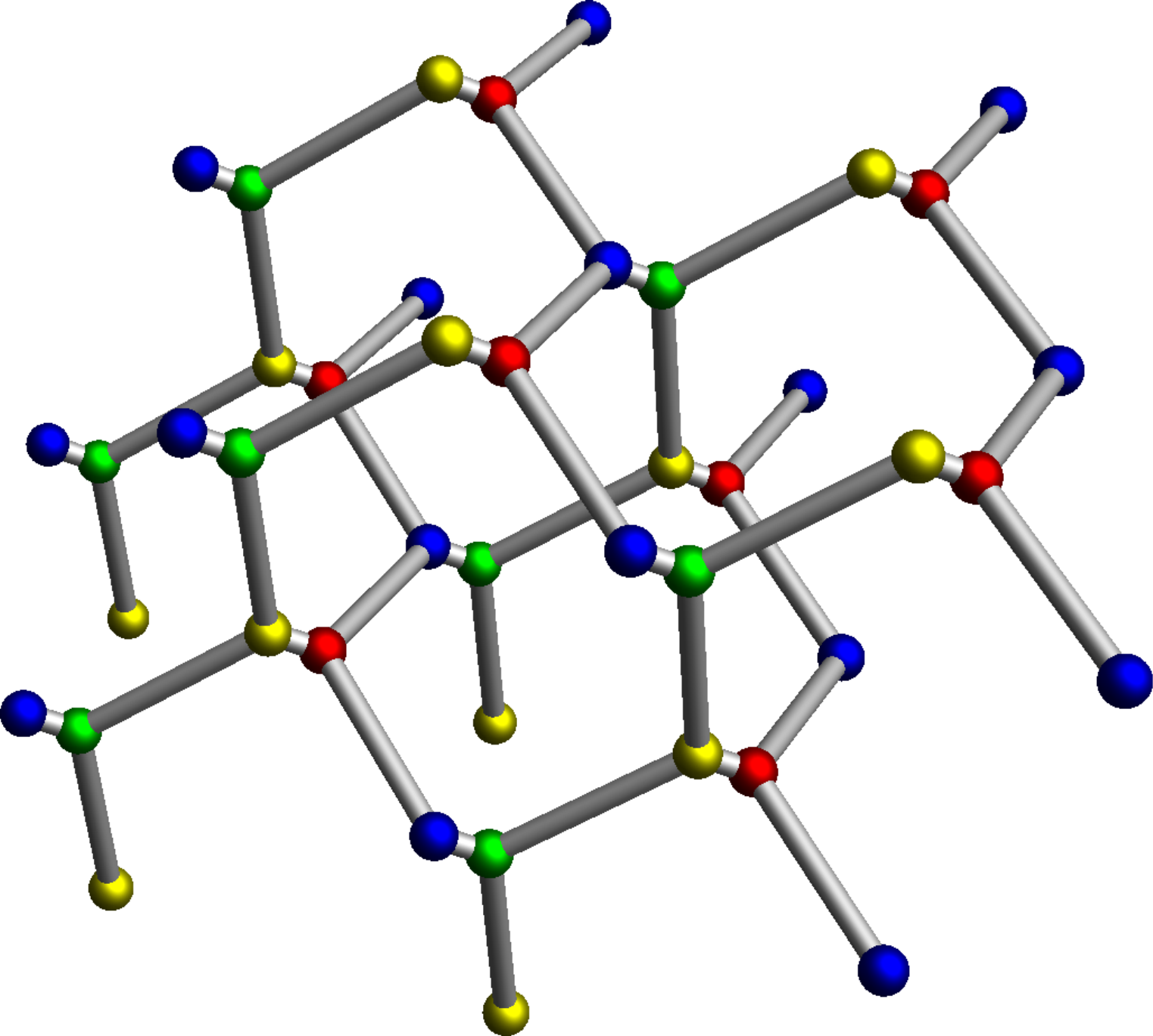}
	\end{subfigure}
	\caption{We identify the \textsf{srs} network (top) as the length minimizer in the class of all triply periodic networks. As indicated by the colouring, 
the quotient has four vertices and is the graph~$K_4$.  Triply periodic Steiner networks on four vertices can also have the graph~$D_1\gsquare D_2$ as a quotient; a minimizing \ths~network is depicted on the bottom.  Observe that the long edges define zigzag curves which are contained in perpendicular planes.
The short edges are contained in lines of intersection of these planes.}
	\label{fig:3D}
\end{figure}

This question is significant in the context of surface theory.
Indeed, the gyroid minimal surface has been compared numerically 
with other explicitely known minimal surfaces.
Compared with these particular surfaces, 
% and under suitable normalizations,
the gyroid has certain optimal features:
among them network length, surface area per fundamental cell, 
Gauss curvature variance, or channel diameter variance 
\cite{kgbgyroidinterface,SchroderTurk2006}.
Nevertheless, it seems completely out of reach to show optimality
in comparison to arbitrary surfaces,
in particular without prescribing a lattice.

The present paper identifies the \srs~network as the unique network 
minimzing length in a sense we describe now, 
and which is illustrated by Figure~\ref{fig:Steinerex2d}
for the two-dimensional case.
Let $\Lambda$ be the lattice of a triply periodic network~$N\subset\R^3$. 
Then the fundamental domain $\R^3/\Lambda$ is a flat $3$-torus 
with volume~$V$, and the network quotient~$N/\Lambda$ has a length~$L$.
We usually refer to $V$ and~$L$ 
as the volume and length of the network~$N$ itself.
Since scaling can reduce the length of~$N$, a well-posed variational problem is:
\emph{Minimize the network length~$L$ under the constraint $V=1$.}
Equivalently, one can minimize the scale-invariant ratio~$L^3/V$.
Our main result answers the above question:
\begin{figure}
	\centering
	\begin{subfigure}[c]{.3\linewidth}
		\centering
		\begin{tikzpicture}[scale=1.5]
			\coordinate[dot1](p0) at (0,0);
			\coordinate[dot1](p1) at (1,0);
			\coordinate[dot1](p2) at (1,1);
			\coordinate[dot1](p3) at (0,1);
			\draw(p0) -- (p1) -- (p2) -- (p3) -- (p0) -- (p2);
			\draw[cross] (p1) -- (p3);
		\end{tikzpicture}
	\end{subfigure}
	\begin{subfigure}[c]{.3\linewidth}
		\centering
		\begin{tikzpicture}[scale=1.5]
			\coordinate[dot1](p0) at (0,0);
			\coordinate[dot1](p1) at (1,0);
			\coordinate[dot1](p2) at (1,1);
			\coordinate[dot1](p3) at (0,1);
			\draw(p0) to [out=-45,in=225] (p1);
			\draw(p0) to [out=45,in=135] (p1);
			\draw(p3) to [out=-45,in=225] (p2);
			\draw(p3) to [out=45,in=135] (p2);
			\draw(p0) -- (p3) (p1) -- (p2);
		\end{tikzpicture}
	\end{subfigure}
	\caption{The two graphs with degree $3$ on four vertices without loops:
          $K_4$ (left) and~$D_1\gsquare D_2$~(right).}
	\label{fig:abstractlist3}
\end{figure}
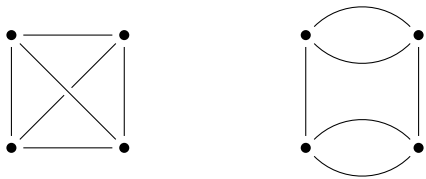

\begin{maintheorem}\label{thm:A}
  The length and volume of a triply periodic network~$N$ in~$\R^3$ satisfy
  % with lattice~$\Lambda$
  \begin{equation}\label{lengthratioestimate}
    \frac {L^3}V\geq\frac{27}{\sqrt2}\, =\, 19.09\ldots\; .
  \end{equation}
  Equality holds exactly for the \srs\ network, 
  where the lattice $\Lambda$ is body-centered cubic.
\end{maintheorem}
\noindent
Here, the terminlogy is as in Section~\ref{sec:stnetworks};
in particular uniqueness is always up to (unnecessary) vertices of degree~$2$.

Let us indicate how the results of our paper combine to prove Theorem~\ref{thm:A}.
First, for a fixed lattice we establish 
the existence of a minimizer of~$L^3/V$ in Theorem~\ref{th:minimizer}.
It must be an embedded Steiner network on $4$~vertices. 
Since a minimizer cannot have loops in the quotient graph 
(Lemma~\ref{lem:Steinerloop})
only the two graphs shown in Figure~\ref{fig:abstractlist3} can arise.
Theorems~\ref{thm:ths} and~\ref{thm:srs} then give
sharp estimates for the ratio $L^3/V$ of networks with arbitrary lattice
for the two cases of quotient graph.
These estimates imply~\eqref{lengthratioestimate} as well as
the characterization of the equality case.
This establishes Theorem~\ref{thm:A} for arbitrary 
triply periodic networks.

An obvious approach to prove the estimates 
of Theorems~\ref{thm:ths} and~\ref{thm:srs}
would be to minimize the ratio $L^3/V$ for a given lattice,
and thereafter minimize over all lattices. 
However, Steiner networks for a given lattice are not unique, 
and lattices are inconvenient to parameterize. 
So we use a different approach:
In Lemmas~\ref{lem:thsformula} and~\ref{lem:srsformula} 
we show that it is possible to parameterize the space of Steiner networks 
covering each of the underlying graphs by their six edge lengths alone
(plus an angle parameter in one case).
Then not only the length~$L$
but also the volume~$V$ become explicit functions of these parameters
(Lemma~\ref{lem:thsformula} and~\ref{lem:srsformula}).
Thus to prove Theorems~\ref{thm:ths} and~\ref{thm:srs} 
we need to solve
a finite dimensional optimization problem under constraints:
On our parameter space, we minimize the total length~$L$ 
under the constraints that $V=1$ and the length parameters be positive.
Then it turns out that lattice generators are linear 
functions of our parameters.

The graph different from $K_4$ which arises in the proof of Theorem~\ref{thm:A}
is~$D_1\gsquare D_2$, the Cartesian graph product
of the dipole graphs $D_1$ and~$D_2$ (see Figure~\ref{fig:abstractlist3}).
% We will not make use of this combinatorial structure.  
In Theorem~\ref{thm:ths} we determine the $1$-parameter family 
which attains the minimal length ratio for this quotient graph 
(see Figure~\ref{fig:3D}).
We call these minimizers \ths\ networks,
again making use of a crystallographic name, this time refering
to the thorium silicide $\mathrm{ThSi_2}$ crystal (it is also known as $(10,3)$-b~\cite{wells}).
Our result for networks with quotient~$D_1\gsquare D_2$ is: 
% The length of optimal networks of type~\ths\ happens to be
% just 2\% larger than the \srs~network, assuming the volume is fixed to~$1$:
% The figure is 1.0198.  However, L^3/V is about 6% larger.
\begin{maintheorem}\label{th:B}
        A triply periodic network in $\R^3$ 
        with quotient~$D_1\gsquare D_2$ satisfies
	\[\frac{L^3}V\geq\frac{81}4\,=\,20.25\, .\]
        Equality holds exactly for a one-parameter family of (non-similar)
        Steiner networks, all with the same lattice~$\Lambda$,
        for instance generated by $(1,0,0),(0,1,0),(1/2,1/2,\sqrt3/2)$. 
        Moreover, each of these minimizing \ths\ networks can be 
        homotoped into a Steiner network of smaller length 
        covering~$K_4$,
        such that the length is non-increasing and the lattice remains fixed.
\end{maintheorem}\noindent
The homotopy is established in Theorem~\ref{th:homotopy}
and visualized in Figure~\ref{fig:thsk4}.
We take it as evidence for why in nature
\ths\ networks play a minor role compared to~\srs. 
\smallskip

\begin{figure}
	\centering
	\begin{subfigure}{.25\linewidth}
		\centering
		\includegraphics[width=\linewidth]{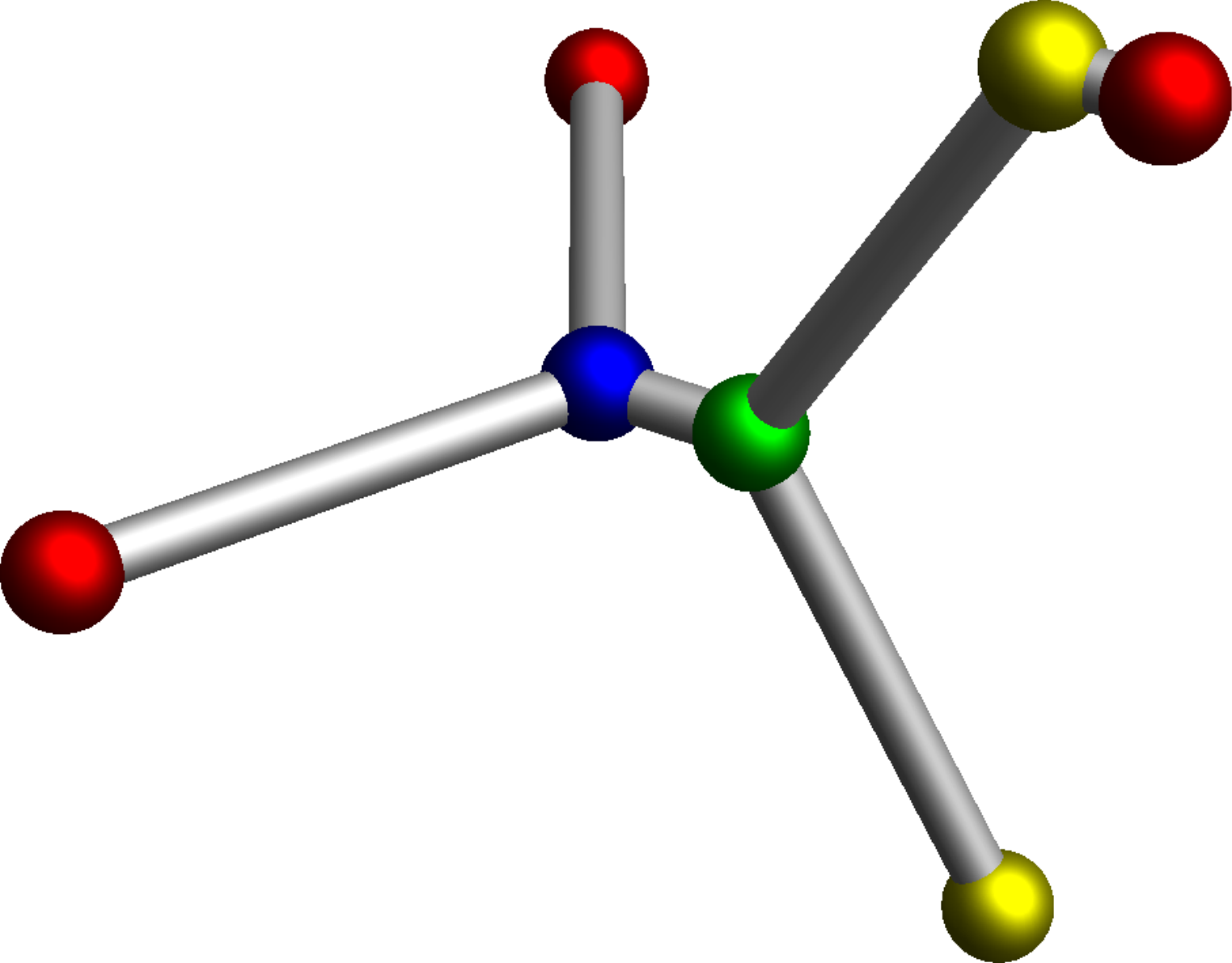}
		\caption{}
	\end{subfigure}\hfill
	\begin{subfigure}{.25\linewidth}
		\centering
		\includegraphics[width=\linewidth]{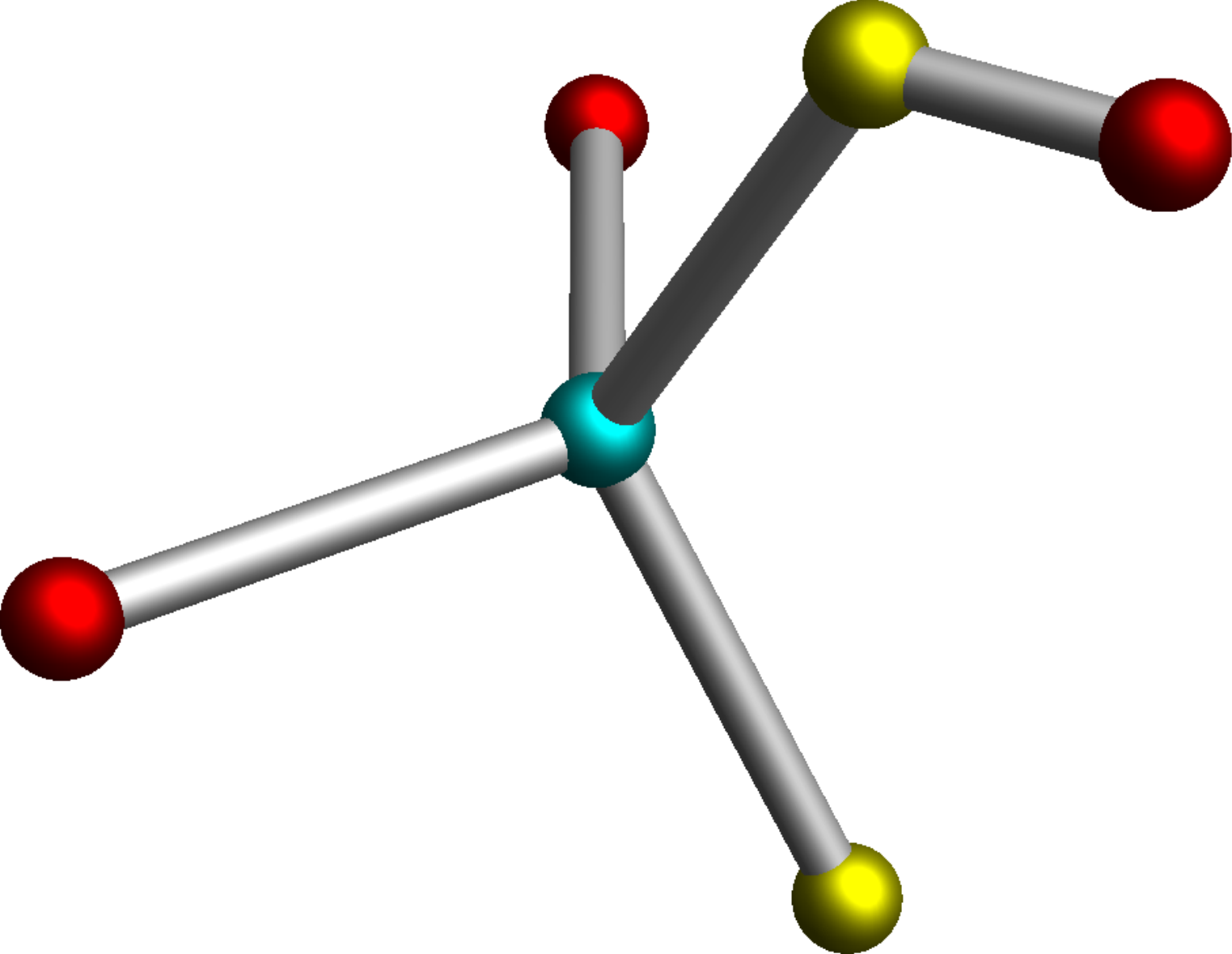}
		\caption{}
	\end{subfigure}\hfill
	\begin{subfigure}{.25\linewidth}
		\centering
		\includegraphics[width=\linewidth]{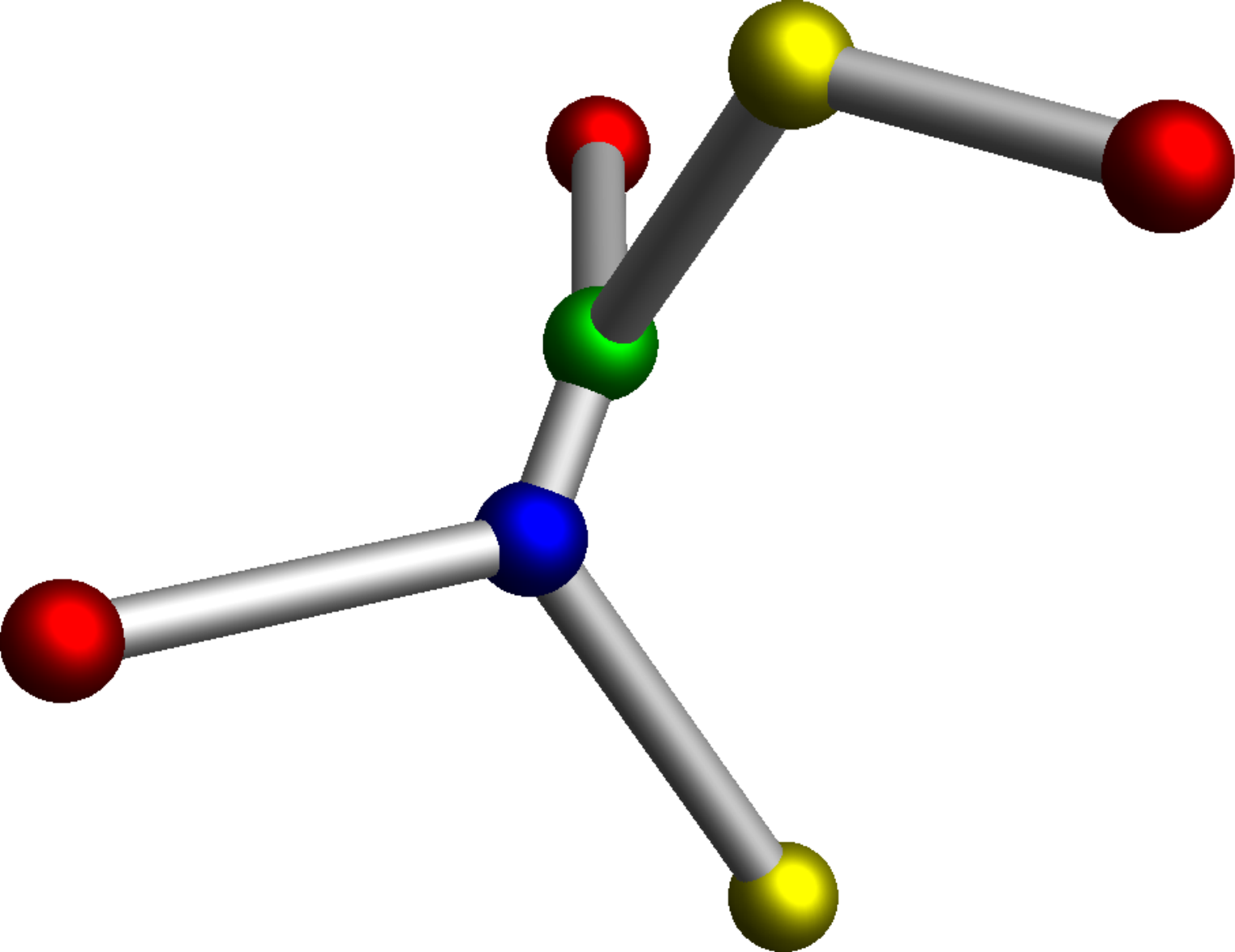}
		\caption{}
	\end{subfigure}
	\caption{A \ths~network~(a) can continuously be deformed into a triply periodic \srs-network~(c).  The homotopy preserves the lattice but decreases length. The change of topology occurs when two vertices coincide~(b).}
	\label{fig:thsk4}
\end{figure}

Our results should be compared with the work of 
Sunada and Kotani~\cite{kotani2001standard} 
(see also~\cite{sunada2012topological,crystals}).
Instead of minimizing the length $L=\sum_{i=1}^m x_i$
of a network with edge lengths $x_1,\ldots,x_m$,
they minimize the quadratic energy $E=\sum_{i=1}^m (x_i)^2$;
they also impose a volume constraint.
For various combinatorial types of networks, 
Sunada and Kotani determine energy minimizers,
essentially by solving a linear algebra problem.
Unlike for our setting, minimizers are unique for given combinatorics. 
The \srs\ network is also the energy minimizer,
while for the case of~$D_1\gsquare D_2$, 
Sunada and Kotani exhibit a unique energy 
minimizer, not contained in our length minimizing \ths~family.

The energy~$E$ can be understood to model a physical crystal 
by harmonic oscillators along the edges of the network.
Nevertheless, for the modelling of many other real world problems, 
in particular those relating to surface theory, 
the network length~$L$
seems more appropriate than the quadratic energy~$E$. 
Note the analogy to the case of curves and surfaces, 
where variational problems involving length and area are natural,
while the minimization of energy is usually simpler to handle mathematically.
At this place we would 
like to recommend Sunada's book~\cite{sunada2012topological} 
as a comprehensive introduction to the 
geometric theory of networks and graphs, 
in particular to a systematic treatment of their topology. 
%
% Here is a short summary of Sunada.
% 1. Approach: 
% a) Take a finite graph, select n generators of homology H_1(G,Z) 
% (perhaps there are more).
% Define an abstract infinite abelian covering graph, by mapping the 
% n generators to nonclosed cycles.  Call this a topological crystal.
% b) Embed this graph into R^n in a non-degenerate way
% (def'd on p.99: no two vertices coincide, no two edges at a vertex 
% point in the same direction).
% 2. Variational problems:
% a) Define harmonic realization of top. cryst. w.r.t. lattice \rho (p.104/05):
% Prescribe the map from the n generators to lattice generators,
% and minimize the energy \sum edge-length^2 extending over quotient edges.
% Euler equation is  \sum edge_vectors = 0 , summing over edges at each vertex,
% that is, harmonic means barycentric vertex placement.
% Note that applying affine maps preserves this notion, that is, a harmonic 
% realization for one lattice gives harmonic realization for all others.
% b) Now consider energy times 1/vol^{n/2}, that is, a normalized energy.
% A standard realization is a minimizer of this energy (p.109), 
% so it also selects a lattice.  Euler equation: additionally satisfies (7.9).
% The standard realization is unique up to similarity (Thm. 7.5).
% In fact, he makes it up to motion, by setting the constant c in (7.9) to 2.
% It can be found by applying linear maps to harmonic realizations,
% and solving for (7.9).

It remains open to determine optimal Steiner networks 
in higher dimensions $n\geq 4$.  
While our reasing generalizes in principle,
the number of admissible combinatorial graphs
strongly increases with~$n$.
In the forthcoming paper \cite{higherdegree} we pursue another direction:
There we study length minimizing $n$-periodic networks 
with vertices of prescribed degree $d\ge 4$ in~$\R^3$ and~$\R^n$.
% or in other ambient spaces
\smallskip

The results presented in this paper are part of the first author's 
PhD thesis in progress.

%%%%%%%%%%%%%%%%%%%%%%%%%%%%%%%%%%%%%%%%%%%%%%%%%%%%%%%%%%%%%%%%%%%%%%%%
\section{Steiner networks}\label{sec:stnetworks}

	\begin{figure}
		\begin{subfigure}[t]{.33\linewidth}
			\centering
			\begin{tikzpicture}[x=.6cm,y=.6cm,rotate=-30]
				\foreach \x in {-1,0,1} {%
					\foreach \y in {-1,0,1} {%
						\begin{scope}[shift = {(\x*1.5,\x*.866+\y*2*.866)}]
							\draw(0,0) -- (-.5,-.866) -- ++(-.4,0);
							\draw(0,0) -- (-.5,.866);
							\draw(0,0) -- (1,0);					
						\end{scope}
					}
				}
				\draw[dashed,fill,fill opacity=.25,color=gray](-.5,-.866) -- ++(1.5,.866) -- ++(0,1.732) -- ++(-1.5,-.866) -- cycle;
				\draw[very thick](-.5,-.866) -- (0,0) -- (-.5,.866) (0,0) -- (1,0);
			\end{tikzpicture}
		\end{subfigure}
		\hfill
		\begin{subfigure}[t]{.3\linewidth}
			\centering
			\begin{tikzpicture}[x=.4cm,y=.4cm,rotate=-19.1]
				\foreach \x in {-1,0,1} {%
					\foreach \y in {-1,0,1} {%
						\begin{scope}[shift = {(\x*2.5+\y*-.25,\x*.866+\y*2.5*.866)}]
							\draw(0,0) -- (-.5,-.866) -- ++(-.4,0);
							\draw(0,0) -- (-.75,1.5 * .866);
							\draw(0,0) -- (2,0);					
						\end{scope}
					}
				}
				\draw[dashed,fill,fill opacity=.25,color=gray](-.5,-.866) -- ++(2.5,.866) -- ++(-.25,2.5 * .866) -- ++(-2.5,-.866) -- cycle;
				\draw[very thick](-.5,-.866) -- (0,0) -- (-.75,1.5 * .866) (0,0) -- (2,0);
			\end{tikzpicture}
		\end{subfigure}
		\hfill
		\begin{subfigure}[t]{.31\linewidth}
			\centering
			\begin{tikzpicture}
				\foreach\x in {0,1}{%
					\foreach\y in {-1,0,1}{%
						\foreach\p in {-1,1}{%
							\begin{scope}[x=.5cm,y=\p*.5cm,xshift=\x*2cm,yshift=\y*.866cm]
								\draw(-.6,0) -- ++(.6,0) -- ++(.5,.866) -- ++(2,0) -- ++(.5,-.866) -- ++(.4,0);
							\end{scope}
						}
					}
				}
				\begin{scope}[x=.5cm,y=.5cm]
					\draw[dashed,fill,fill opacity=.25,color=gray](1.5,-.433) -- ++(4,0) -- ++(0,1.732) -- ++(-4,0) -- cycle;
					\draw[very thick](1.5,.866) -- ++(1,0) -- +(.25,.433) +(0,0) -- ++(.5,-.866) -- +(-.25,-.433) +(0,0) -- ++(1,0) -- +(.25,-.433) +(0,0) -- ++(.5,.833) -- +(-.25,.433) +(0,0) -- ++(1,0);
				\end{scope}
			\end{tikzpicture}
		\end{subfigure}
		\caption{Doubly periodic Steiner networks and their fundamental domains.  The underlying abstract graph of the first two networks is the dipole graph of order~$3$ (cf.~Figure~\ref{fig:hcb}).
The first network, with the hexagonal lattice, minimizes length for given area of its fundamental domain.  The network on the right has the quotient~$D_1\gsquare D_2$.}
		\label{fig:Steinerex2d}
	\end{figure}
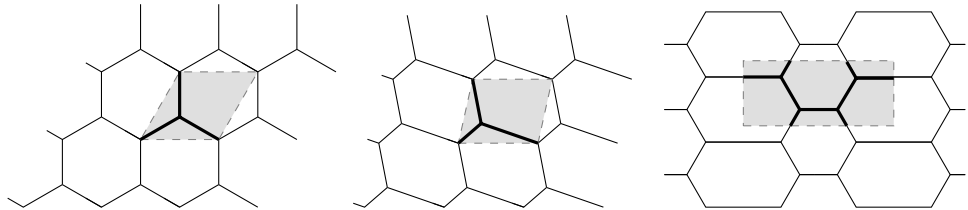
For our purposes, it is convenient to use the term network in the
following sense:
\begin{definition}
	An \emph{($n$-periodic) network $N$} is a connected simple graph, 
% "simple": 
% unweighted, undirected graph containing no graph loops or multiple edges 
immersed with straight edges of positive length into~$\R^n$, subject to the following:
	\begin{itemize}
	\item
	$N$ is invariant under a lattice~$\Lambda$ 
        of rank~$n$.
	\item
	The quotient $N/\Lambda$ is a finite graph~$\Gamma$,
	possibly with loops and multiple edges.  
	\end{itemize}
	We call $V=V(\R^n/\Lambda)$ the \emph{(spanned) volume} of $N$ 
  and $L=L(N/\Lambda)$ the \emph{length} of~$N$.
\end{definition}
% Terminology (see wikipedia Graph_(discrete_mathematics) --> link glossary)
% incident: The relation between an edge and one of its endpoints.
% adjacent: The relation between two vertices that are both endpoints 
% of the same edge.
%
Let us explain our terminology. 
If a graph is mapped injectively to~$\R^n$ 
then we call the map an \emph{embedding}. 
% Since the quotient N/\Lambda is compact 
% an embedding is an injective immersion.
We call it an \emph{immersion} if for each vertex
the restriction to the union of the incident edges 
is injective, i.e., the star of each vertex is embedded. 
% Sunada uses another definition of immersion:
% Merely the directions (say: unit vectors) of the incident edges
% at a vertex must be distinct.
% This allows for an edge to cover itsself, unlike our def.
% However, we rule out selfintersections anyway, so the notion
% of immersedness is purely technical.

A \emph{lattice} (\emph{of rank} $n$) is a set $\Lambda= \{\sum_{i=1}^n a_ig_i : a_i\in\mathbb Z\}\subset\R^n$, where the vectors $g_1,\ldots,g_n\in\R^n$ are linearly independent. The ambient space quotient $\R^n/\Lambda$ is an $n$-dimensional flat torus.  It can be represented by a parallelepiped spanned by the vectors~$g_i$. We refer to the quotient or its representing epiped as a \emph{fundamental domain}. 

Conversely, we can start with an abstract finite graph~$\Gamma$. 
Our networks can then be described as immersions 
of certain abelian coverings of~$\Gamma$.  
For the cases considered in the present paper, 
the first homology group $H_1(\Gamma,\mathbb Z)$
has $n$ closed cycles as generators (that is, the first Betti number is~$n$).
Thus for $N$ to be $n$-periodic, each such cycle must map to a generator 
of the lattice.
See Sunada~\cite{sunada2012topological} for a detailed account of
the covering theory of graphs. 

We wish to minimize the length~$L=L(N/\Lambda)$ 
of the quotient network $N/\Lambda$, 
subject to the constraint that the $n$-dimensional volume $V=V(\R^n/\Lambda)$ 
of a fundamental domain is fixed to~$1$. Equivalently, we may minimize 
the scaling-invariant \emph{length ratio}~$L^n/V$.

Suppose a vertex $p$ of a network is connected with edges to the
three vertices $q_1,q_2,q_3$.  If the network is critical for length
the first variation formula gives 
\begin{align}\label{eq:equilibrium}
	\frac{p-q_1}{\vert p-q_1\vert}+\frac{p-q_2}{\vert p-q_2\vert}
  +\frac{p-q_3}{\vert p-q_3\vert} = 0\,.
\end{align}
Equivalently, the three edges incident to~$p$ meet at $120^\circ$-angles.
We refer to \eqref{eq:equilibrium} 
as the \emph{Steiner condition} or as \emph{balancing}.
\begin{definition}
	A \emph{Steiner network} is an $n$-periodic network 
	where all vertices have degree~$3$ (the network is $3$-\emph{regular})
  and satisfy the Steiner condition~\eqref{eq:equilibrium} at each vertex.
\end{definition}
Steiner networks result from minimizing the length ratio:
\begin{theorem}\label{th:minimizer}
  Among $n$-periodic networks in $\R^n$ with fixed lattice~$\Lambda$, 
  there exists a network minimizing the length~$L$. %ratio~$L^n/V$. 
  Any such minimizing network~$N$ is an embedded Steiner network 
  such that $N/\Lambda$ has $2n-2$~vertices,
  up to vertices of degree~$2$ with opposite incident edges. 
\end{theorem}
\begin{remark*}
If a minimizer has vertices of degree $2$ then the incident edges must 
be opposite, and so such vertices can always be removed without changing~$L$.
From now on we assume this is the case.
\end{remark*}
For the proof, we need the notion of the \emph{circuit rank} 
% also called cyclomatic number, circuit number
of a connected finite graph~$\Gamma$,
\[
  \rank\Gamma:= 1 -\#\text{vertices of $\Gamma$} +\#\text{edges of $\Gamma$}\,.
\]
Note that a tree has circuit rank~$0$; that for any connected graph
the circuit rank is a non-negative integer; 
and that for a $3$-regular graph we have
\begin{align}\label{eq:rankfordegreethree}
  \rank \Gamma = 1 + \frac12\, \#\text{vertices of $\Gamma$} .
\end{align}
The circuit rank is precisely 
the number of generators of $H_1(N/\Lambda,\mathbb Z)$.
% also called fundamental cycles.
To verify this, consider a spanning tree $T\subset N/\Lambda$ 
of a quotient network $N/\Lambda$, so that $H_1(T,\mathbb Z)$ is trivial.
Reinsert the edges one by one to see that both the circuit rank of~$T$, 
as well as the number of generating cycles in~$T$,
increases by~$1$ in each step.

For a network~$N$, we define the circuit rank 
as the rank of its quotient, $\rank N:=\rank(N/\Lambda)$.
The following two lemmas serve to show that we can assume
a minimizing sequence of networks to 
have rank~$n$ and be $3$-regular.

\begin{lemma} \label{le:circuitrank}
  Let $N$ be an $n$-periodic network with lattice~$\Lambda\subset\R^n$.
  If $\rank N>n$ then there exists an $n$-periodic network $N'$,
  also with lattice $\Lambda$, 
  which has smaller length, $L(N')<L(N)$, and $\rank N'=n$.
\end{lemma}
\begin{proof}
	We construct graphs $G_0\subset\cdots\subset G_n$ such that $\rank G_i=i$.
	The graph $G_0$ is chosen as a spanning tree 
  of the quotient network $N/\Lambda$,
  while for $i=1,\dots,n$ the graph $G_i$ is the union of~$G_{i-1}$ 
	with a single edge $e_i\in (N/\Lambda)\setminus G_{i-1}$, 
  subject to the requirement that $G_i$ lifts to a network $N_i\subset N$
  with a lattice of rank~$i$.   
  Then $N':=N_n$ has circuit rank~$n$ and is $n$-periodic.
  Since $\rank N>n$, the network~$N'$ 
	has fewer edges than~$N$, and so $L(N')<L(N)$.
\end{proof}

\begin{lemma} \label{le:impossibledegrees}
  Suppose an $n$-periodic network $N$ with lattice $\Lambda$ contains a 
  vertex with degree $d\geq 4$ or $d=1$,
  or with degree~$d=2$ and non-opposite edges.
  Then there exists an $n$-periodic network~$N'$ with lattice~$\Lambda$
  and $\rank N'=\rank N$, such that $N'$ is $3$-regular 
  with smaller length, $L(N')<L(N)$.
\end{lemma}
\begin{proof}
  Clearly we can decrease length 
  by successively removing all vertices of degree~$d=1$ from~$N/\Lambda$
  together with their incident edges.
  If the resulting network contains a vertex of degree~$2$
  with non-opposite edges we can replace these edges 
  by a single edge to reduce length.

  If the resulting graph contains a vertex~$p$ 
  of degree~$d\geq4$ we use a well-known argument to reduce length
  (see for instance~\cite[{p.\,120f.}]{minimalnetworks}). 
  The star at $p$ contains two (non-collinear) edges with endpoints $q_1,q_2$
	which make an angle of less than $120$~degrees.
  To define~$N'$, 
  we replace these two edges in $N/\Lambda$ with a tripod which connects
	the triple $p,q_1,q_2$ with a further point in the same plane,
  % which can be, for instance, the Steiner point of the triple.
  chosen such that the length decreases.
  Thereby the degree at~$p$ changes from $d$ to~$d-1$.
  Upon iteration we can reduce the degree to~$d\leq3$ at all vertices
  of~$N/\Lambda$.

  Our operations preserve $n$-periodicity.  However, 
  for $d=2$ and $d\geq4$ they possibly do not preserve the immersion property.
  So assume incident to~$p$ there are two or more edges in the same direction.
  We replace the initial portion, up to the first vertex,
  by a single edge, thereby reducing length.
  Iteration of this construction 
  yields a network of shorter length, all of its stars are embedded.

  Finally, we merge opposite edges incident to a vertex of degree~$2$ 
  to form a single edge, leaving $L$ unchanged,
  and let $N'$ be the resulting network.
  Note that all our operations preserve the circuit rank.
  Hence $N'$ has the properties claimed.
\end{proof}

\sloppypar
\begin{proof}[Proof of the theorem]
  Consider a length minimizing sequence $(N_k)$,
  that is, $\lim_{k\to\infty} L(N_k)=\inf L(N)$, where the infimum is
  taken over all $n$-periodic networks with fixed lattice~$\Lambda$.
  
  Applying Lemma \ref{le:circuitrank} we may assume that $\rank N_k=n$,
  and applying Lemma~\ref{le:impossibledegrees} thereafter we may assume
  $N_k$ still has $\rank N_k=n$ but is $3$-regular.
  By~\eqref{eq:rankfordegreethree} 
  the number of vertices then is~$2n-2$.
  Combinatorially, there are only finitely many such graphs.
  Thus by passing to a subsequence we can assume 
  all $N_k/\Lambda$ have the same combinatorial type~$\Gamma$.

  The set of $2n-2$ vertices of $N_k/\Lambda$ is compact in~$\R^n/\Lambda$,
  and the connecting edges are geodesics with uniformly bounded length.
  Hence vertices and edges of a further subsequence of~$(N_k)$
  converge to a limit~$N$.
  We claim that all edges of~$N$ attain positive length.
  To see this, note that if an edge attains length~$0$, 
  then $N$ has a vertex with degree~$d\geq4$.
  By Lemma~\ref{le:impossibledegrees}, however, $N$ cannot be a minimizer of 
  length over all combinatorial types,
  contradicting the fact that $(N_k)$ is a minimizing sequence.  

  Let us show $N$ is embedded. 
  Suppose two edges intersect in an interval of positive length. 
  Then comparison with a network where the intersection set is replaced
  by a single edge shows that $N$ cannot be minimizing.
  Similarly, supposing $N$ has an isolated point of intersection 
  we compare $N$
  with a network where this point is a vertex of degree~$d\geq4$.
  Then $N$ cannot minimize by Lemma~\ref{le:impossibledegrees}.
% Note: If two or more edges intersect at interior points, 
% then this means to split the edges by introducing a new vertex,
% necessarily of degree at least 4.
% If the intersection point includes at least one vertex,
% then splitting edges which contain this point, or merging 
% this vertex with the other vertices at the same location 
% results in a degree at least 5 (= 3+2).
  Finally, since $N$ is a minimizer with positive edge lengths 
  the first variation formula~\eqref{eq:equilibrium} shows $N$ is Steiner.
\end{proof}

\begin{remark*}
The proof indicates that our results do not change 
if we drop the connectivity assumption in the definition of
$n$-periodic networks (but still require that the cycles of the underlying 
possibly disconnected graph span a lattice of rank~$n$).
Indeed, if a minimizer was disconnected, we could use translations to
move one component as to intersect the other. 
Again this contradicts Lemma~\ref{le:impossibledegrees}.
\end{remark*}

We now show that Steiner networks cannot contain loops, 
thereby constraining the combinatorial types further.
For instance the graphs shown in Figure~\ref{fig:graphswithloops} 
are impossible for networks critical for length with~$n=3$.
\begin{lemma}\label{lem:Steinerloop}
  Let $N\subset\R^n$ be an $n$-periodic Steiner network.  
  Then $N/\Lambda$ contains no loops.
\end{lemma}
\begin{proof}
  A loop based at a vertex $p\in N/\Lambda$ corresponds to a straight edge 
  of the lift~$N$.  Thus the lift contains a vertex with opposite incident edges,
  thereby contradicting balancing.
\end{proof}
With Proposition~\ref{prop:Steinerdouble} we will derive yet another 
constraint on the combinatorial graph of a length minimizing Steiner network:
it must be simple, i.e., it cannot contain double edges.
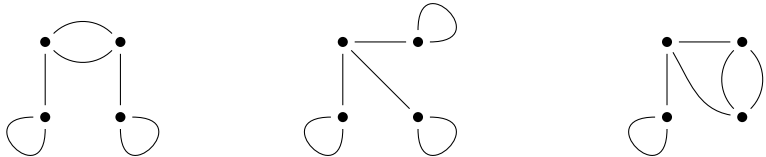
\begin{figure}
	\hspace*{\fill}
	\begin{subfigure}[t]{.26\linewidth}
		\centering
		\begin{tikzpicture}
			\coordinate[dot1](p0) at (0,0);
			\coordinate[dot1](p1) at (1,0);
			\coordinate[dot1](p2) at (1,1);
			\coordinate[dot1](p3) at (0,1);
			\draw(p0) to [out=180,in=-90,looseness=10] (p0);
			\draw(p1) to [out=-90,in=0,looseness=10] (p1);
			\draw(p3) to [out=-45,in=225] (p2);
			\draw(p3) to [out=45,in=135] (p2);
			\draw(p0) -- (p3) (p1) -- (p2);
		\end{tikzpicture}
	\end{subfigure}\hfill
	\begin{subfigure}[t]{.26\linewidth}
		\centering
		\begin{tikzpicture}
			\coordinate[dot1](p0) at (0,0);
			\coordinate[dot1](p1) at (1,0);
			\coordinate[dot1](p2) at (1,1);
			\coordinate[dot1](p3) at (0,1);
			\draw(p0) to [out=180,in=-90,looseness=10] (p0);
			\draw(p1) to [out=-90,in=0,looseness=10] (p1);
			\draw(p2) to [out=0,in=90,looseness=10] (p2);
			\draw(p0) -- (p3);
			\draw(p1) -- (p3);
			\draw(p2) -- (p3);
		\end{tikzpicture}
	\end{subfigure}\hfill
	\begin{subfigure}[t]{.26\linewidth}
		\centering
		\begin{tikzpicture}
			\coordinate[dot1](p0) at (0,0);
			\coordinate[dot1](p1) at (1,0);
			\coordinate[dot1](p2) at (1,1);
			\coordinate[dot1](p3) at (0,1);
			\draw(p0) to [out=180,in=-90,looseness=10] (p0);
			\draw(p0) -- (p3) (p3) -- (p2);
			\draw(p3) to [out=-60,in=170] (p1);
			\draw(p1) to [out=45,in=-45] (p2);
			\draw(p1) to [out=135,in=-135](p2);
		\end{tikzpicture}
	\end{subfigure}\vspace*{-.5cm}\hspace*{\fill}
	\caption{All connected $3$-regular graphs with loops on $4$~vertices. By Lemma~\ref{lem:Steinerloop}, none of these graphs can be the quotient graph of a minimizer.}
    \label{fig:graphswithloops}
\end{figure}
\begin{remark*}
	The combinatorial graph of a minimizer may depend on the lattice. 
  Indeed, as a result of \cite{higherdegree} 
  triply periodic networks with two degree-$5$~vertices which minimize length
  have different combinatorial types for different prescribed lattices.
% The minimizer over all lattices is~\bnn, 
% while if the lattice is fixed to coincide with the lattice of 
% the length minimizing \sqp~network, 
% then in fact \sqp\ (with a different quotient graph) 
% is the absolute minimizer. 
%% Its combinatorial graph is~$D_{1,3}$. 
  However, for the Steiner case, the homotopy of Theorem~\ref{th:homotopy}
  implies that a minimizer for a fixed lattice always has $K_4$ as a quotient.
\end{remark*}

%%%%%%%%%%%%%%%%%%%%%%%%%%%%%%%%%%%%%%%%%%%%%%%%%%%%%%%%%%%%%%%%%%%
\section{Maclaurin's inequality for elementary symmetric polynomials}

In the cases we will consider, the volume $V$ of a given network $N$ 
with $m$ labelled edges of length $(x_1,\ldots,x_m)$ 
is a polynomial $P(x_1,\ldots,x_m)$.
Thus the task to minimize the quotient $L^3/V$ is equivalent to maximizing the polynomial~$P$ under the length constraint $L=1$,
where $L(x_1,\ldots,x_m):=x_1+\ldots+x_m$.

In the most symmetric case, 
$P$ is the \emph{elementary symmetric polynomial} of degree~$k$,
\[P_k(x_1,\ldots,x_m) := \sum_{1\leq j_1<\ldots<j_k\leq m}x_{j_1}\cdots x_{j_k}\,,\qquad 1\leq k\leq m\,.\]
We also set $P_0(x_1,\ldots,x_m):=1$.  We can estimate these polynomials 
by the length:
\begin{lemma}[Maclaurin's inequality]\label{le:elsympol}
	If $x_i\geq0$ for all $1\leq i\leq m$ and $k\geq2$ then
	\begin{align}
		P_k(x_1,\ldots,x_m)\leq\binom mk\Bigl(\frac{x_1+\ldots+x_m}m\Bigr)^k\,,\label{eq:maclaurin}
	\end{align}
	where equality holds if and only if $x_1 = \ldots = x_m$.
\end{lemma}
In particular, for degree $k\geq2$ the elementary symmetric polynomial $P_k$ takes its maximum under the length constraint $L=1$ 
exactly at $(\frac1m,\ldots,\frac1m)$. 
One way to prove Maclaurin's inequality is to use Newton's inequality,
see~\cite{hardy1952inequalities}.
We present a more direct proof here, inspired by our application.
\begin{proof}
	We prove \eqref{eq:maclaurin} by induction over $m$. The base case is $m=k$, where $P_k = x_1\cdots x_m$. Then \eqref{eq:maclaurin} is the estimate on geometric and arithmetic mean.
	
	For the step suppose $m>k\geq2$. 
We claim \eqref{eq:maclaurin} holds strictly if some but not all $x_i$ vanish.  In view of the symmetry of~\eqref{eq:maclaurin} we may assume $x_m=0$. Note that $P_k(x_1,\ldots,x_{m-1},0)$ is an elementary symmetric polynomial of degree $k$ in $m-1$ variables, and so the induction hypothesis gives
	\begin{align*}
		P_k(x_1,\ldots,x_{m-1},0) \leq \binom{m-1}k\Bigl(\frac{x_1+\ldots+x_{m-1}}{m-1}\Bigr)^k\,.
	\end{align*}
	We estimate the right hand side, using the strict inequality
\[
  \binom{m-1}k\frac1{(m-1)^k} 
  = \frac1{k!}\,\frac{m-1}{m-1}\cdots\frac{m-k}{m-1}
  \,< \,\frac1{k!}\, \frac mm\cdots\frac{m-k+1}m = \binom mk\frac1{m^k}\,.
\]
% The sum contains at least two terms, and strict inequality holds 
% for all factors, except for the first:
% Indeed, (m-j)/(m-1) < (m-j+1)/m  <=> m^2-mj < m^2-mj+m -m+j-1  ok for j\ge 2,
If not all $x_i$ vanish this yields strict inequality in~\eqref{eq:maclaurin},
as claimed.
	
Since~\eqref{eq:maclaurin} is scaling invariant, 
it is sufficient to prove this inequality under the length constraint~$L=1$. 
The continuous function $P_k$ attains a maximum over the compact set $L^{-1}(1)\subset[0,\infty)^m$ at some point $z=(z_1,\ldots,z_m)$.
Note that $z\not=0$, that
we have equality in~\eqref{eq:maclaurin} for $x=(\frac1m,\ldots,\frac1m)$,
and that the induction hypothesis, in form of the claim,
gives strict inequality in~\eqref{eq:maclaurin}
on $\partial([0,\infty)^m)\setminus\{0\}$.
% In fact, at the origin 0, we also have equality
Thus $z$ must be an interior point of $[0,\infty)^m$. 
Since $z$ is critical for~$P_k$ under the smooth constraint~$L=1$
we obtain the necessary condition
\begin{align} \label{lagrangeforPk}
  \nabla P_k(z) = \lambda\nabla L(z)=\lambda(1,\ldots,1)\,
\end{align}
with $\lambda\in\R$ a Lagrange multiplier.
It remains to show this implies $z_1=\ldots=z_m$. 
Then since $z$ assigns equality to~\eqref{eq:maclaurin} and $z$ was chosen maximally, 
the proof of the induction step is completed.

Since $P_k$ is elementary symmetric, 
for $i\neq j$ we can express $P_k$ at any point $x=(x_1,\ldots,x_m)$ as
\[P_k(x) = x_ix_jQ_0(x) + (x_i+x_j)Q_1(x) + Q_2(x)\,,\]
where $Q_0,Q_1,Q_2$ are polynomials in $m-2$ variables,
independent of $x_i$ and~$x_j$. 
From \eqref{lagrangeforPk} we conclude 
$\partial_iP_k(z) = \partial_jP_k(z)$ for all $1\leq i,j\leq m$, so that
	\[z_jQ_0(z) + Q_1(z) = z_iQ_0(z) + Q_1(z)\,.\]
Moreover, since $k\geq2$ and $z_i>0$ for all~$i$, the polynomial $Q_0$ cannot vanish at~$z$. Thus indeed $z_i=z_j$ for all~$i,j$. 
\end{proof}

%%%%%%%%%%%%%%%%%%%%%%%%%%%%%%%%%%%%%%%%%%%%%%%%%%%%%%%%%%%%%%%%%%%
\section{Doubly periodic Steiner networks}\label{sec:dimension2}

We find it instructive to present the case of dimension $n=2$ 
before studying the more involved case $n=3$. 
We first determine the topology of the quotient graph
of a minimizer for prescribed lattice.
By Theorem~\ref{th:minimizer} it has $2$~vertices,
% after removal of unnecessary degree~$2$ vertices - standing assumption!
and by Lemma~\ref{lem:Steinerloop} it has no loops. 
% regular graph: same degree at each vertex
The only connected $3$-regular graph on $2$~vertices without loops is~$D_3$,
see Figure~\ref{fig:hcb}.
Hence we obtain:
\begin{lemma}\label{lem:mintwod}
  A doubly periodic network $N\subset\R^2$ minimizing the length ratio $L^2/A$ 
  for prescribed lattice~$\Lambda$ is Steiner on $2$~vertices
  with the dipole graph~$D_3$ as a quotient. 
\end{lemma}
	\begin{figure}[b]
		\centering
		\begin{subfigure}[c]{.4\linewidth}
			\centering
				\begin{tikzpicture}
					\coordinate[dot1](p0) at (0,0);
					\coordinate[dot2](p1) at (3,0);
					\draw[steiner1](p0) -- (p1);
					\draw[steiner2](p0) to [in=-135,out=-45] (p1);
					\draw[steiner3](p0) to [in=135,out=45] (p1);
				\end{tikzpicture}
		\end{subfigure}
		\begin{subfigure}[c]{.4\linewidth}
			\centering
			\begin{tikzpicture}
				\coordinate[dot1](p0) at (0,0);
				\coordinate[dot2](p1) at (0:2);
				\coordinate[dot2](p2) at (120:1.5);
				\coordinate[dot2](p3) at (240:1);
				\draw[steiner2](p0) -- node[left ]{$x_3$} (p3);
				\draw[steiner3](p0) -- node[left ]{$x_2$} (p2);
				\draw[steiner1](p0) -- node[below]{$x_1$} (p1);
				\draw[steiner3](p1) -- ($(p1)-.7*(p2)$);
				\draw[steiner2](p1) -- ($(p1)-.8*(p3)$);
				\draw[steiner2](p2) -- ($(p2)-.7*(p3)$);
				\draw[steiner3](p3) -- ($(p3)-.7*(p2)$);
				\draw[steiner1](p3) -- ($(p3)-.5*(p1)$);
				\draw[steiner1](p2) -- ($(p2)-.6*(p1)$);
				\node[above right] at (p0) {$p_0$};
				\node[right] at (p1) {$p_1$};
				\node[above left] at (p2) {$p_2$};
				\node[below left] at (p3) {$p_3$};
			\end{tikzpicture}
		\end{subfigure}
		\caption{The dipole graph~$D_3$ and a network covering it.}
		\label{fig:hcb}
	\end{figure}
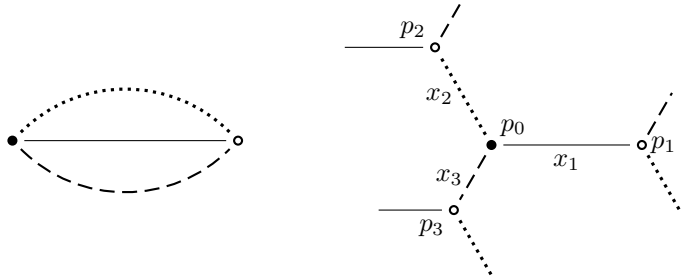
Since the edges of a Steiner network enclose $120^\circ$-angles,
a minimizing network can be described in terms of three edge lengths alone:
\begin{lemma}\label{lem:hcbformula}
  Up to isometry, a doubly periodic Steiner network~$N\subset\R^2$
  with quotient $D_3$ 
  is uniquely determined by its three edge lengths $x_1,x_2,x_3>0$. 
  Its length and spanned area are
\begin{align*}
  L = x_1+x_2+x_3 \qquad\text{and}\qquad 
  A = \frac{\sqrt3}{2}(x_1x_2+x_1x_3+x_2x_3)\,.
\end{align*}
\end{lemma}
\begin{proof}
	The two vertices of $D_3$ correspond to a vertex $p_0\in N$ 
  and the incident vertices $p_1,p_2,p_3\in N$, where the labelling
  relates to the lengths as in Figure~\ref{fig:hcb}.
  Then the lattice $\Lambda$ of $N$ is spanned, for instance, by 
  $g_1 := p_1 - p_3$ and $g_2 := p_2 - p_3$.
  Specifically, we may assume that up to isometry we have
	\begin{align*}
		p_0 &= \begin{pmatrix}0\\0\end{pmatrix}, & p_1 &=x_1\begin{pmatrix}1\\0\end{pmatrix}, & p_2 &= x_2\begin{pmatrix}-\frac12\\\frac{\sqrt3}{2}\end{pmatrix}, & p_3 &= x_3\begin{pmatrix}-\frac12\\-\frac{\sqrt3}{2}\end{pmatrix},
	\end{align*}
  and so the lattice $\Lambda$ has area
\[
  A = \left\vert\det(g_1,g_2)\right\vert 
    = \frac{\sqrt3}{2}(x_1x_2+x_1x_3+x_2x_3)\,.
\]
\vspace{-10mm}

\end{proof}
As might be expected, the optimal doubly periodic 
network is given by the tesselation of $\R^2$ with regular hexagons:
\begin{proposition}\label{prop:hcb}
	For each doubly periodic network $N\subset\R^2$ we have
	\begin{align} \label{eqn:estimatetwod}
		\frac{L^2}{A}\geq 2\sqrt3\,.
	\end{align}
	Equality holds if $N$ has the quotient $D_3$ and the three edge lengths of $N$
  are equal; then the lattice is hexagonal.
\end{proposition}
% Does a hexagonal lattice and Steiner imply equality?
% If we work in the class of immersions (or embeddings): yes
% If we work with Sunada's def. of immersion: no 
% Sunada allows for for edges to overlap itsself.  For the hexagonal
% lattice this can for instance be realized by joining only the next 
% but one vertex of the lattice with an edge.
\begin{proof}
  For a prescribed lattice $\Lambda$, 
  Lemma~\ref{lem:mintwod} asserts the existence of a Steiner network~$N_0$ 
  with quotient~$D_3$ which is a minimizer, $(L^2/A)(N)\ge (L^2/A)(N_0)$,
  where the inequality is strict unless $N$ has quotient~$D_3$.
  According to Lemma~\ref{lem:hcbformula},
  the edge lengths $x_1,x_2,x_3>0$ determine~$N_0$,
  and the area~$A(N_0)$ is a multiple 
  of the elementary symmetric polynomial of degree~$2$ in three variables.
  Thus Maclaurin's inequality~\eqref{eq:maclaurin} 
  implies \eqref{eqn:estimatetwod} for~$N_0$:
  \[
    A(N_0)  = \frac{\sqrt3}2(x_1x_2+x_1x_3+x_2x_3)
    \leq\frac{3\sqrt3}2\Bigl(\frac{x_1+x_2+x_3}3\Bigr)^2 
    = \frac1{2\sqrt3}(L(N_0))^2\,
  \]
  In particular, \eqref{eqn:estimatetwod} follows for~$N$.
  To discuss the equality case, note that for~$N$ with quotient~$D_3$ 
  and $x_1=x_2=x_3$, equality in \eqref{eqn:estimatetwod} is obvious.  
  But by the above and Lemma~\ref{le:elsympol} 
  the equality can only hold for this case.
\end{proof}

%%%%%%%%%%%%%%%%%%%%%%%%%%%%%%%%%%%%%%%%%%%%%%%%%%%%%%%%%%%%%%%%%%%
\section{Triply periodic Steiner networks covering \texorpdfstring{$D_1\gsquare D_2$}{D1 D2}}

Our approach to triply periodic Steiner networks is similar to the doubly periodic case. 
However, as pointed out in the Introduction, Theorem~\ref{th:minimizer} 
and Lemma~\ref{lem:Steinerloop} allow exactly two 
distinct topologies of minimizing Steiner networks: 
\begin{lemma}
  The combinatorial graph of a triply periodic Steiner network,
  minimizing length for a prescribed lattice~$\Lambda$,
  is~$K_4$ or~$D_1\gsquare D_2$.
\end{lemma}
We analyze the case $D_1\gsquare D_2$ first since our analysis of the 
more prominent $K_4$-case makes use of it.
In both cases we can parameterize the space of networks by 
the edge lengths $x_1,\ldots,x_6$ of the six 
edges $e_1,\ldots,e_6$ in the quotient $N/\Lambda$; 
for $D_1\gsquare D_2$~there is a further angle parameter.
This will follow from considering the tangent planes
at the vertices; note that the Steiner condition implies 
that each vertex is coplanar with its three neighbours. 
In dimension $n=3$, the angles between the different tangent planes 
turn out to be independent of the edge lengths.
\smallskip

With respect to a labelling as in Figure~\ref{fig:ths} we state:
% note that the single edges are labelled $e_5$ and~$e_6$.
	\begin{figure}
		\centering
		\begin{subfigure}[t]{.45\linewidth}
			\centering
			\begin{tikzpicture}[x=2.6cm,y=2.6cm]
				\coordinate[dot1](p0) at (0,0);
				\coordinate[dot2](p1) at (1,0);
				\coordinate[dot4](p2) at (1,1);
				\coordinate[dot3](p3) at (0,1);
				\draw[steiner2](p0) to [out=45,in=135] node[above]{$e_2$} (p1);
				\draw[steiner1](p0) to [out=-45,in=-135] node[below]{$e_1$} (p1);
				\draw[steiner5](p3) to [out=45,in=135] node[above]{$e_4$} (p2);
				\draw[steiner4](p3) to [out=-45,in=-135] node[below]{$e_3$} (p2);
				\draw[steiner3](p0) to node[left ]{$e_5$} (p3);
				\draw[steiner6](p1) to node[right]{$e_6$} (p2);
			\end{tikzpicture}
		\end{subfigure}
		\begin{subfigure}[t]{.45\linewidth}
			\centering
			\begin{tikzpicture}[x=1.7cm,y=1.7cm]
				\coordinate[dot2](p0) at (-.5,-.866);
				\coordinate[dot1](p1) at (0,0);
				\coordinate[dot3](p2) at (.6,0);
				\coordinate[dot4](p3) at (1.1,.866);
				\coordinate[dot2](p4) at ($1.4*(-.5,.866)$);
				\coordinate[dot4](p5) at ($(.6,0)+.8*(.5,-.866)$);
				\coordinate[dot2](p6) at (2.1,.866);
				\draw[steiner1](p0)--node[left]{$x_1$}(p1);
				\draw[steiner2](p1)--node[left]{$x_2$}(p4);
				\draw[steiner3](p1)--node[above]{$x_5$}(p2);
				\draw[steiner5](p2)--node[left]{$x_4$}(p5);
				\draw[steiner4](p2)--node[right]{$x_3$}(p3);
				\draw[steiner6](p3)--node[above]{$x_6$}(p6);
				\draw[steiner5](p3)--+($.4*(p4)$);
				\node[below] at (p0) {$p_0$};
				\node[left] at (p1) {$p_1$};
				\node[right] at (p2) {$p_2$};
				\node[left] at (p3) {$p_3$};
				\node[above] at (p4) {$p_4$};
				\node[below] at (p5) {$p_5$};
				\node[below] at (p6) {$p_6$};
			\end{tikzpicture}
		\end{subfigure}
		\caption{The graph~$D_1\gsquare D_2$ and its covering network schematically.}
		\label{fig:ths}
	\end{figure}
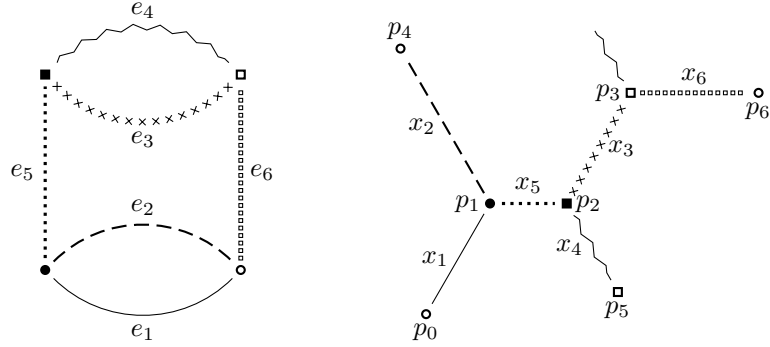
\begin{lemma}\label{lem:thsformula}
	Let $N\subset\R^3$ be a triply periodic Steiner network with 
  quotient $D_1\gsquare D_2$.  Then, up to isometry, the network $N$ 
  is uniquely determined by its six edge lengths $x_1,\ldots,x_6>0$ 
  and an angle $\alpha\in(0,\pi)$. 
  Moreover, $N$ has length $L=\sum_i x_i$ and, for a labelling of the
  edge lengths as in Figure~\ref{fig:ths}, the spanned volume is 
	\begin{align}
		\begin{split}
	V =\frac34\sin(\alpha)\big(x_1x_2x_3&+x_1x_2x_4+x_1x_3x_4+x_2x_3x_4\\
			&+(x_5+x_6)(x_1x_3+x_2x_3+x_1x_4+x_2x_4)\big)\,.
		\end{split}
		\label{eq:thsvol}
	\end{align}
\end{lemma}
% The volume $V$ happens to be the sum of all possible products of three edge 
% lengths where at least two edges are not adjacent 
% and two edges correspond to a double edge in $\Gamma$. 
\noindent
We will see that all edges of $N$ are contained in two sets of parallel planes
which make an angle~$\alpha$ to be chosen independently of the edge lengths. 
The limiting cases $\alpha=0$ and~$\pi$ relate to a doubly periodic network.
\begin{proof}
Consider a connected subgraph $\tilde N\subset N$ with seven vertices $p_0,\ldots,p_6$ as in Figure~\ref{fig:ths} such that $p_0,p_4,p_6$, as well as $p_3,p_5$ are identified in the quotient. We may assume $p_1$ is the orgin, $p_2$ is on the $x$-axis, and $p_0,p_4$ are in the $xy$-plane. Balancing then implies
\begin{align}\label{eq:thspoints0124}
  p_0=x_1\begin{pmatrix}-\frac12\\-\frac{\sqrt3}2\\0\end{pmatrix}\,,\quad 
  p_1=\begin{pmatrix}0\\0\\0\end{pmatrix}\,,\quad 
  p_2=x_5\begin{pmatrix}1\\0\\0\end{pmatrix}\,,\quad 
  p_4=x_2\begin{pmatrix}-\frac12\\\frac{\sqrt3}2\\0\end{pmatrix}\,.
\end{align}
The tangent plane at $p_2$ must be a rotation about the $x$-axis 
of the tangent plane at $p_1$ by an angle~$\alpha\in[0,2\pi)$. 
Let $A_\alpha\in SO(3)$ denote such a rotation. 
% i.e., A_\alpha = (1 0 0 // 0 cos\alpha -sin\alpha // 0 \sin\alpha \cos\alpha)
The Steiner condition then implies that 
$p_3-p_2$ points in the same direction as $p_1-p_0$ rotated by $A_\alpha$.
The same applies to $p_5-p_2$ and $p_1-p_4$.  That is,
\begin{equation}\label{eq:thspoints35}
\begin{split}
  % (p_3-p_2)/x_3 = A_\alpha (p_1-p_0)/x_1
  % <=>  p_3 = p_2 + x_3 A_\alpha (-p_0/x_1)
  p_3 &= p_2 + x_3\,A_\alpha\!\begin{pmatrix}\frac12\\\frac{\sqrt3}2\\0\end{pmatrix}
  =\frac12\begin{pmatrix}x_3+2x_5\\
                x_3\sqrt3\cos\alpha\\x_3\sqrt3\sin\alpha\end{pmatrix}\,,\\
  % vector is (p_1-p_4)/x_2 
  p_5 &= p_2 +x_4\,A_\alpha\!\begin{pmatrix}\frac12\\-\frac{\sqrt3}2\\0\end{pmatrix}  =\frac12\begin{pmatrix}x_4+2x_5\\
               -x_4\sqrt3\cos\alpha\\-x_4\sqrt3\sin\alpha\end{pmatrix}\,.
\end{split}
\end{equation}
Triple periodicity implies $\alpha\neq0\bmod \pi$ and changing $\alpha$ to $\alpha\pm\pi$ corresponds to a change of numbering of the vertices $p_3$ and $p_5$.
% Note that we have not used x_6 yet.
So we may assume $\alpha\in(0,\pi)$. Using the Steiner condition we see that for a pair of vertices which are doubly connected in $N/\Lambda$ the normals must agree. Since $p_6$ and $p_0$ are identified in $N/\Lambda$ the vector $p_6-p_3$ points in the same direction as $p_2-p_1$. So we have
\begin{equation}\label{eq:thspoints6}
  p_6 = p_3 + x_6\begin{pmatrix}1\\0\\0\end{pmatrix}=\frac12\begin{pmatrix}x_3+2x_5+2x_6\\x_3\sqrt3\cos\alpha\\x_3\sqrt3\sin\alpha\end{pmatrix}\, .
\end{equation}
The three vectors
\begin{equation}\label{eq:thslattice}
\begin{split}
	g_1 &:= p_4 - p_0 = \frac12\begin{pmatrix}x_1-x_2\\\sqrt3(x_1+x_2)\\0\end{pmatrix}\,,\\
	g_2 &:= p_5 - p_3 = \frac12\begin{pmatrix}x_4-x_3\\-\sqrt3(x_3+x_4)\cos\alpha\\-\sqrt3(x_3+x_4)\sin\alpha\end{pmatrix}\,,\\
	g_3 &:= p_6 - p_0 = \frac12\begin{pmatrix}x_1+x_3+2x_5+2x_6\\\sqrt3(x_1+x_3\cos\alpha)\\x_3\sqrt3\sin\alpha\end{pmatrix}
\end{split}
\end{equation}
span the lattice $\Lambda$; indeed, an inspection of Figure~\ref{fig:ths} shows they correspond to minimal cycles in the abstract graph. 
Then $\vert\det(g_1,g_2,g_3)\vert$ can be computed to~\eqref{eq:thsvol}.
% Note: Terms to sum must be of Form x_ix_jx_k, symmetric in x_1,...,x_4.
% The result depends on x_5,x_6 only through  x_5+x_6 
% Perhaps this is good enough to determine the formula for V.
\end{proof}
\begin{figure}[b]
	\centering
	\begin{subfigure}[c]{.45\linewidth}
		\centering
		\begin{tikzpicture}[scale=.7]
			\coordinate[dot3](r) at (-1,-2);
			\coordinate[dot1](p) at (-1,0);
			\coordinate[dot2](q) at (1,0);
			\coordinate[dot4](s) at (1,-2);
			\draw[steiner5](r) -- (p);
			\draw[steiner1](p) to [out=45,in=135] (q);
			\draw[steiner2](p) to [out=-45,in=225] (q);
			\draw[steiner3](q) -- (s);
			\node[left] at (p) {$p$};
			\node[right] at (q) {$q$};
			\node[left] at (r) {$r$};
			\node[right] at (s) {$s$};
		\end{tikzpicture}
	\end{subfigure}
	\begin{subfigure}[c]{.45\linewidth}
		\centering
		\begin{tikzpicture}
			\coordinate[dot1](p0) at (0,0);
			\coordinate[dot2](q-1) at (.5,.866);
			\coordinate[dot2](q0) at (.5,-.866);
			\coordinate[dot1](p1) at (0,-1.732);
			\coordinate(p-1) at (0,1.732);
			\coordinate(q1) at (.5,-2.6);
			\coordinate[dot3](r0) at (-1.2,0);
			\coordinate[dot4](q'-1) at (1.2,.866);
			\coordinate[dot4](s0) at (1.2,-.866);
			\coordinate[dot3](p'1) at (-1.2,-1.732);
			\draw[steiner1]($(p-1)!.5!(q-1)$) -- (q-1);
			\draw[steiner2](q-1) -- (p0);
			\draw[steiner1](p0) -- (q0);
			\draw[steiner2](q0) -- (p1);
			\draw[steiner1](p1) -- ($(q1)!.5!(p1)$);
			\draw[steiner3](q-1) -- (q'-1);
			\draw[steiner5](p0) -- (r0);
			\draw[steiner3](q0) -- (s0);
			\draw[steiner5](p1) -- (p'1);
			\node[right] at (p0) {$p_0$};
			\node[left] at (q0) {$q_0$};
			\node[left] at (r0) {$r_0$};
			\node[right] at (s0) {$s_0$};
		\end{tikzpicture}
	\end{subfigure}
	\caption{A Steiner network with double edges. The stars of the two doubly connected vertices lie in a common plane.}
	\label{fig:Steinerdouble}
\end{figure}
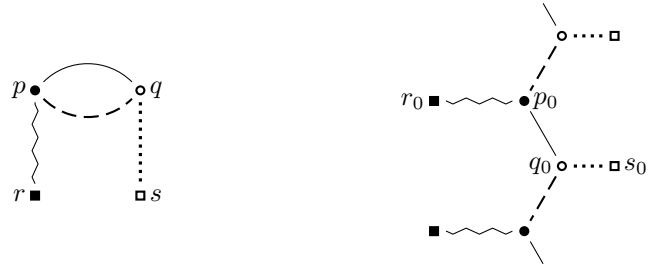

As an aside, we use the reasoning of Lemma~\ref{lem:thsformula} to show
that a minimizer $N\subset\R^n$ of~$L^n/V$ for prescribed lattice~$\Lambda$ 
cannot contain double edges for~$n\geq3$.
According to Theorem~\ref{th:minimizer}
the network~$N$ is an $n$-periodic Steiner network.
Let $p_0,q_0\in N$ be two adjacent vertices,
and suppose they project onto doubly connected vertices $p,q\in N/\Lambda$.
Denote by $r_0$ the neighbour of~$p_0$ which does not project to~$q$, and by $s_0$ the neighbour of~$q$ not projecting to~$p$, see Figure~\ref{fig:Steinerdouble}. 
Then the Steiner condition shows the vectors $r_0-p_0$ and $s_0-q_0$ 
are parallel and point into opposite directions. 
Now move $p_0$ and~$q_0$ simultanously in one of these directions:
For $0\leq t<1$, replace $p_0$ by $p^t_0:=p_0+ t(r_0-p_0)$ 
and $q_0$ by $q^t_0:=q_0+ t(r_0-p_0)$, 
and similarly so for all other lifts of~$p,q$.
We obtain an $n$-periodic Steiner network~$N^t$ 
with the same lattice~$\Lambda$ and $L(N_t) = L(N)$.  
The limiting network~$N_1$ with lattice~$\Lambda$
has length $L(N_1)=L(N)$ and so is again minimizing.  
However, $N_1$ has one vertex of degree~$4$, 
thereby contradicting Lemma~\ref{le:impossibledegrees}.
Our reasoning proves:
\begin{proposition}\label{prop:Steinerdouble}
  An $n$-periodic minimizer of $L^n/V$ with $n\geq3$ for prescribed lattice 
  covers a simple graph on $2n-2$ vertices of degree~$3$.
\end{proposition}

\begin{remark*}
The number of connected $3$-regular simple graphs 
on $2n-2$ vertices, i.e., 
cubic graphs, is rapidly growing in~$n\geq3$, 
see \href{http://oeis.org/A002851}{oeis.org}.
% m=4: 2 graphs (6 vertices):- 
% - maximally symmetric hexagon, 
%   with diagonals which give minimal cycles with 4 vertices
% - a hexagon with two kindes of vertices, such that 
%   one cycle of 4 vertices and two cycles with 3 vertices
% m=5: 5 graphs, m=6: 19 graphs
\end{remark*}

The proposition implies that a triply periodic 
minimizer can only have the quotient~$K_4$.
Thus if we are merely interested in establishing Theorem~\ref{thm:A}
it may appear that we do not need the estimate 
for the quotient~{$D_1\!\gsquare D_2$}, stated in the next theorem.
However, a limiting case of \eqref{eq:ths} below
will enter the proof of Theorem~\ref{thm:srs}, 
and the equality result will also be used in Section~\ref{sec:homotopy}.

To determine optimal networks with quotient~{$D_1\!\gsquare D_2$}
we now solve a standard calculus problem, 
namely we maximize the function~$V$ under a constraint for~$L$.
Interestingly enough, up to similarity of~$\R^3$ there is 
a one-parameter family of optimal networks, 
meaning that these networks are not strictly stable: 
\begin{theorem}\label{thm:ths}
  Let $N\subset\R^3$ be a triply periodic Steiner network with 
  quotient~\mbox{$D_1\!\gsquare D_2$}. Then
\begin{align}\label{eq:ths}
	\frac{L^3}V\geq\frac{81}4\,,
\end{align}
where equality holds if and only if
\begin{align}\label{eq:thsequality}
	x_1=x_2=x_3=x_4=2x_5+2x_6\,\quad\text{and}\quad\alpha=\frac\pi2\,.
\end{align}
In the equality case the lattice is generated, up to similarity, by
$(0,1,0)$, $(0,0,1)$, $\frac12(\sqrt3,1,1)$.
\end{theorem}
\begin{proof}
  Admitting vanishing edge lengths, we will show the inequality in a 
  form implying \eqref{eq:ths}, namely
  \begin{align}\label{eq:thsinequality}
    V\leq\frac4{81}L^3\quad\text{ for all $x\in[0,\infty)^6$ and $\alpha\in(0,\pi)$}\,,
  \end{align}
  with equality precisely for \eqref{eq:thsequality}.

  For fixed $x=(x_1,\ldots,x_6)$ clearly $L$ is independent of~$\alpha$, 
  while~\eqref{eq:thsvol} gives that $V$ is maximal exactly at $\alpha=\pi/2$.
  % thereby confirming the last condition claimed in \eqref{eq:thsinequality}
  Moreover, both $V$ and~$L$ depend on $x_5,x_6$ only through $y:=x_5+x_6$. 
  Thus in order to establish \eqref{eq:thsinequality} 
  we may fix $\alpha$ to~$\pi/2$
  and consider the functions induced by $L$ and~$V$ 
  on the domain $[0,\infty)^5\ni(x_1,x_2,x_3,x_4,y)$.
  For the remainder of the proof 
  we denote these continuous functions again by $L$ and~$V$.

  We claim that \eqref{eq:thsinequality} holds along the boundary 
  of~$[0,\infty)^5$.  Trivially, this is true at~$0$.
  Otherwise let $(x_1,x_2,x_3,x_4,y)$ be a point 
  where at least one coordinate vanishes. In case $y=0$ the volume is
  \[ 
    V =\frac34\bigl(x_1x_2x_3+x_1x_2x_4+x_1x_3x_4+x_2x_3x_4\bigr)\,.
  \]
	The right-hand side contains the elementary symmetric polynomial of degree $k=3$ in $m=4$ variables and so indeed, by Maclaurin's inequality~\eqref{eq:maclaurin},
	\[
    V \leq3\Bigl(\frac{x_1+x_2+x_3+x_4}4\Bigr)^3 = \frac3{64}L^3 < \frac4{81}L^3\,.
  \]
  % Indeed, the strict inequality holds since we assume L\not=0.
	The other case is that some $x_i$ vanishes for $i=1,2,3$, or~$4$. 
  In view of the symmetry of $V$ and~$L$ it suffices to consider the case $x_1=0$. 
Under this assumption 
Maclaurin's inequality gives
	\[V=\frac34 x_2(x_3x_4+x_3y+x_4y)\leq\frac14x_2(x_3+x_4+y)^2\,.\]
Then the claim follows from the estimate on geometric and arithmetic mean,
\[V\leq x_2\,\frac{x_3+x_4+y}2\;\frac{x_3+x_4+y}2\leq\frac1{27}(x_2+x_3+x_4+y)^3=\frac1{27}L^3<\frac4{81}L^3\,.\]
	
We now proceed as in the proof of Maclaurin's inequality. 
The continuous function~$V$ attains its maximum 
on the compact set $L^{-1}(1)\subset[0,\infty)^5\setminus\{0\}$
at some point $z:=(x_1,\ldots,x_4,y)$.

One easily verifies $V=4\,L^3/81$ if \eqref{eq:thsequality} holds. 
Thus our claim implies that in fact $z\in (0,\infty)^5$.
For the set $(0,\infty)^5$, 
the point $z$ is critical for $V$ under the constraint $L=1$, and so
	\[ \nabla V(z)= \lambda\nabla L(z) \,,\]
where $\lambda\in\R$ is a Lagrange multiplier. Equivalently,
	\begin{align}
   \frac34 \begin{pmatrix}x_2x_3+x_2x_4+x_3x_4+x_3y+x_4y\\
		x_1x_3+x_1x_4+x_3x_4+x_3y+x_4y\\
		x_1x_2+x_1x_4+x_2x_4+x_1y+x_2y\\
		x_1x_2+x_1x_3+x_2x_3+x_1y+x_2y\\
		x_1x_3+x_2x_3+x_1x_4+x_2x_4		
		\end{pmatrix}
  = \lambda\begin{pmatrix}1\\1\\1\\1\\1\end{pmatrix} 
\,.\label{eq:thslagrange}
	\end{align}
We claim this implies $x_1 = x_2 = x_3 = x_4 = 2y$.
For the proof, we consider dot products of \eqref{eq:thslagrange} 
with four different vectors.  Namely, the product 
with $(1,-1,0,0,0)$ gives $(x_2-x_1)(x_3+x_4)=0$,
the product with $(0,0,1,-1,0)$ gives $(x_1+x_2)(x_4-x_3)=0$.
Moreover, for $(0,1,0,-1,0)$ we obtain $(x_4 - x_1)(x_1+x_4+2y)=0$,
and for $(0,1,0,0,-1)$ we obtain $x_1(2y-x_1)=0$.
Since $z$ has positive coordinates our four equations prove the claim.

We have shown there is a unique critical point 
$z\in(0,\infty)^5$ for $V$ under the constraint $L=1$, 
where $V$ attains its maximal value $V(z)=4/81$.
This implies the inequality~\eqref{eq:thsinequality} first for $L=1$, 
and then, by the scaling invariance of $L^3/V$, in general.
Finally, the uniqueness of~$z$ implies that in general equality 
holds if and only if~\eqref{eq:thsequality} holds; 
to verify the lattice vectors use~\eqref{eq:thslattice}.
\end{proof}

%%%%%%%%%%%%%%%%%%%%%%%%%%%%%%%%%%%%%%%%%%%%%%%%%%%%%%%%%%%%%%%%%%%%%%
\section{The \srs\ network covering the \texorpdfstring{$K_4$}{K4} graph}

We discuss the network related to the gyroid.  
Each vertex of a Steiner network has a well-defined affine tangent plane, 
containing the edge vectors to the incident vertices; 
each vertex in $N/\Lambda$ defines a tangent plane up to translation.
(We avoid the usage of normal vectors since the tangent planes are unoriented.)

For a Steiner network with quotient $K_4$ we use balancing and the
fact that each pair of vertices in $K_4$ is connected with an edge
to show that the four tangent planes are perpendicular to the
four space diagonal directions: 
\begin{lemma}\label{lem:srsangle}
  Let $N\subset\R^3$ be a triply periodic Steiner network 
  with quotient graph~$K_4$. 
  Then the four tangent planes of $N/\Lambda$ are tangent to the four
  faces of a regular tetrahedron.
  Consequently, up to isometry of $\R^3$, the network $N$ is uniquely 
  defined by its six edge lengths $x_1,\ldots,x_6>0$. 
\end{lemma}
	\begin{figure}[t]
		\centering
		\begin{subfigure}[t]{.45\linewidth}
			\centering
			\begin{tikzpicture}[x=3cm,y=3cm]
				\coordinate[dot1](p0) at (0,0);
				\coordinate[dot2](p1) at (1,0);
				\coordinate[dot3](p2) at (1,1);
				\coordinate[dot4](p3) at (0,1);
				\draw[steiner5](p3) -- node[pos=.25,above right]{$\!e_5$} (p1);
				\draw[steiner6,cross](p0) -- node[pos=.3,above left]{\;$e_2$} (p2);
				\draw[steiner1](p0) -- node[below]{$e_1$} (p1);
				\draw[steiner2](p1) -- node[right]{$e_6$} (p2);
				\draw[steiner3](p2) -- node[above]{$e_4$} (p3);
				\draw[steiner4](p3) -- node[left ]{$e_3$} (p0);
			\end{tikzpicture}
		\end{subfigure}
		\begin{subfigure}[t]{.45\linewidth}
			\centering
			\begin{tikzpicture}[x=1.5cm,y=1.5cm]
				\coordinate[dot1] (p0) at (0,0);
				\coordinate[dot2] (p1) at (1,0);
				\coordinate[dot3] (p2) at (-.5,.866);
				\coordinate[dot4] (p3) at (-.5*1.3,-.866*1.3);
				\coordinate[dot4] (p4) at ($(p2) + (.8*.5,.8*.866)$);
				\coordinate[dot2] (p5) at ($(p3) + (.8*-1,0)$);
				\coordinate[dot3] (p6) at ($(p1) + (.8*.5,.8*-.866)$);
				\coordinate (p7) at ($(p1) + (.5*.5,.5*.866)$);
				\coordinate (p8) at ($(p3) + (.4*.5,.4*-.866)$);
				\coordinate (p9) at ($(p2) + (-.5,0)$);
				\node[below right] at (p0) {$p_0$};
				\node[right] at (p1) {$p_1$};
				\node[below left] at (p2) {$p_2$};
				\node[right] at (p3){$p_3$};
				\node[right] at (p4) {$p_4$};
				\node[below] at (p5) {$p_5$};
				\node[below] at (p6) {$p_6$};
				\draw[steiner2](p1) -- node[left, pos=.7]{$x_6$} (p6);
				\draw[steiner5](p1) -- (p7);
				\draw[steiner5](p3) -- node[above]{$x_5$} (p5);
				\draw[steiner3](p3) -- (p8);
				\draw[steiner3](p2) -- node[left,pos=.7]{$x_4$} (p4);
				\draw[steiner2](p2) -- (p9);
				\draw[steiner1](p0) -- node[above]{$x_1$} (p1);
				\draw[steiner6](p0) -- node[right]{$x_2$} (p2);
				\draw[steiner4](p0) -- node[left]{$x_3$} (p3);
			\end{tikzpicture}
		\end{subfigure}
		\caption{The graph $K_4$ and the labelling of the network covering it.}
		\label{fig:srs}
	\end{figure}
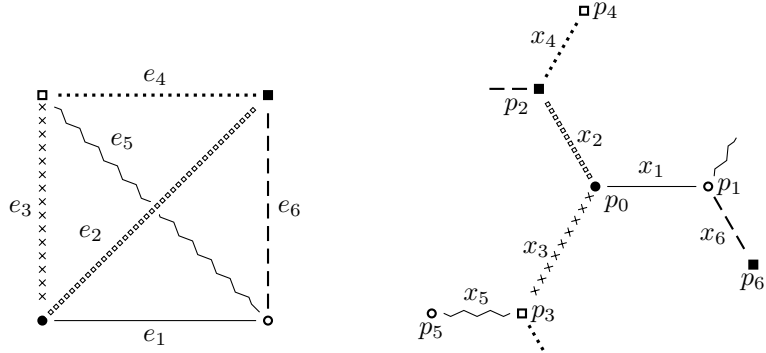%
\begin{proof}
  From $N$ we pick a connected subgraph which contains a vertex $p_0$ 
  and its three neighbours $p_1,p_2,p_3$, 
  representing the vertices of $N/\Lambda$.
  Without loss of generality we may assume $p_0$ to be the origin, 
  $p_1$ to lie on the $x$-axis and $p_2,p_3$ to lie in the $xy$-plane.
  That is, we assume 
	\begin{align}
		p_0 &= \begin{pmatrix}0\\0\\0\end{pmatrix}, & p_1 &=x_1\begin{pmatrix}1\\0\\0\end{pmatrix}, & p_2 &= x_2\begin{pmatrix}-\frac12\\\frac{\sqrt3}{2}\\0\end{pmatrix}, & p_3 &= x_3\begin{pmatrix}-\frac12\\-\frac{\sqrt3}{2}\\0\end{pmatrix},\label{eq:srscoords}
	\end{align}
  where $x_i>0$ is the edge length of the edge incident to~$p_i$.

  Let $p_6\neq p_0$ be a vertex incident to $p_1$, compare Figure~\ref{fig:srs}.
  Copying the reasoning of the proof of Lemma~\ref{lem:thsformula} 
  we find, in terms of some rotation $A_\beta$ 
  about the $x$-axis, where $-\pi<\beta<\pi$:
  \[
    p_6-p_1 = \frac{x_6}{x_2} A_\beta(p_0-p_2) 
    = x_6A_\beta\begin{pmatrix}\frac12\\-\frac{\sqrt3}2\\0\end{pmatrix}   
    = x_6\begin{pmatrix}\frac12\\
        -\frac{\sqrt3}2\cos(\beta)\\-\frac{\sqrt3}2\sin(\beta)\end{pmatrix}\,.
  \]
  Then $\min\big\lbrace\vert\beta\vert,\pi-\vert\beta\vert\big\rbrace$ represents the dihedral angle between 
  the two tangent planes at $p_0$ and~$p_1$.
	
  In the quotient $N/\Lambda$, the vertex $p_6$ must be identified 
  with one of the four vertices $p_0,\ldots,p_3$. %FOUR
  Since the shortest cycle in~$K_4$ consists of three edges 
  this vertex must be either $p_2$ or~$p_3$.  
  Suppose $p_6$ is identified with~$p_2$.
  The tangent planes at these two points agree as vector spaces.
  Hence the balancing equation~\eqref{eq:equilibrium} 
  implies that the vectors $p_2-p_0$ and $p_6-p_1$ enclose $120$~degrees, 
  and the sum of the two unit vectors pointing into these directions 
  must be a unit vector:
  \[
    1 = \left\vert \frac{p_0-p_2}{\vert p_0-p_2\vert} +
                        \frac{p_1-p_6}{\vert p_1-p_6\vert} \right\vert 
    = \left\vert\begin{pmatrix}0\\\frac{\sqrt3}{2} (-1 + \cos (\beta ))\\
                         \frac{\sqrt3}{2}\sin(\beta)\end{pmatrix}\right\vert 
    = \sqrt{\frac32}\sqrt{1-\cos(\beta)}\,.
    % ==> 1-\cos\beta = 2/3 ==> \cos\beta = 1/3
  \]
  The other case is that $p_6$ is identified with~$p_3$. Then, similarly,
  \[
    1=\left\vert\frac{p_1-p_6}{\vert p_1-p_6\vert} 
                           + \frac{p_0-p_3}{\vert p_0-p_3\vert}\right\vert 
    = \left\vert\begin{pmatrix}0\\ \frac{\sqrt3}{2} (1+\cos (\beta ))\\ 
                         \frac{\sqrt3}{2}\sin(\beta)\end{pmatrix}\right\vert 
    = \sqrt{\frac32}\sqrt{1+\cos(\beta)}\,.
    % ==> 1+\cos\beta = 2/3 ==> \cos\beta = -1/3
  \]
  From both cases we conclude $|\cos(\beta)| = 1/3$,
  % i.e., \beta = 70.5... or 109.4... degrees.
  and so the dihedral angle of the tangent planes at $p_0$ and~$p_1$ 
  is the tetrahedral angle $\arccos(1/3)\approx70.53^\circ$.

  In $K_4$, any pair of vertices is connected by an edge, 
  and so the same argument applies to any pair of vertices $p_i,p_j$ 
  of~$N/\Lambda$.  But four planes in~$\R^3$ 
  can only have pairwise dihedral angles~$\arccos(1/3)$ 
  if they are parallel to the faces of a regular tetrahedron.

  Finally, lengths and tangent planes determine a Steiner network completely
  up to isometry.
\end{proof}

For the next statement we choose to label the six edges 
$e_1,\ldots,e_6$,
such that the edges $e_i$ and $e_{i+3}$ do not have endpoints in common,
see Figure~\ref{fig:srs}. We let $x_i$ be the length of~$e_i$.

\begin{lemma}\label{lem:srsformula}
  Let $N$ be a triply periodic Steiner network with quotient~$K_4$. 
  Then $N$ has length $L=\sum_i x_i$ and the spanned volume is
	\begin{align}
\begin{split}
		V =\frac{1}{\sqrt2}&\big(x_1x_2x_4 + x_1x_2x_5 + x_1x_2x_6 + x_1x_3x_4
							\\&+x_1x_3x_5+x_1x_3x_6+x_1x_4x_5+x_1x_4x_6
							+x_2x_3x_4+x_2x_3x_5\\&+x_2x_3x_6+x_2x_4x_5
							+x_2x_5x_6+x_3x_4x_6+x_3x_5x_6+x_4x_5x_6\big)\,. 
		\end{split}
		\label{eq:srsvol}
	\end{align}
\end{lemma}
\noindent
The sum extends over all possible products of three edge lengths
except for those relating to triples of concurrent edges.
\begin{remark*}\label{rem:srsorientation}
  By Lemma~\ref{lem:srsangle}, lengths and tangent planes 
	determine a Steiner network completely up to isometry. 
	Up to rigid motion, however, there are two different Steiner networks 
	covering $K_4$ with the same edge lengths. The isometry mapping the two networks onto another is a reflection which corresponds to a sign change of~$\beta$. Note that the four tangent planes at the vertices of a network are the tangent planes of a regular tetrahedron. Hence, the choice of any two tangent planes determines the other two.
\end{remark*}
\begin{proof}
	After isometry of $\R^3$  we may assume the coordinates are as in~\eqref{eq:srscoords}. For $i=1,2,3$ let $A^i_\beta\in SO(3)$ be the rotation fixing $p_i$ with an angle $\beta = \arccos(1/3)$. In view of Remark~\ref{rem:srsorientation}, possibly by replacing $\beta$ by $-\beta$, the three vectors
	\begin{align*}
		g_1 &:= (p_0 - p_2) + (p_1 - p_0) + \frac{x_6}{x_2}A^1_\beta(p_0 - p_2) = \begin{pmatrix}x_1+\frac12x_2+\frac12x_6\\-\frac{\sqrt3}2x_2-\frac1{2\sqrt3}x_6\\-\sqrt{\frac23}\,x_6\end{pmatrix}\,,\\
		g_2 &:= (p_0 - p_3) + (p_2 - p_0) + \frac{x_4}{x_3}A^2_\beta(p_0 - p_3) = \begin{pmatrix}-\frac12x_2+\frac12x_3\\\frac{\sqrt3}2x_2+\frac{\sqrt3}2x_3+\frac1{\sqrt3}x_4\\-\sqrt{\frac23}\,x_4\end{pmatrix}\,,\\
		g_3 &:= (p_0 - p_1) + (p_3 - p_0) + \frac{x_5}{x_1}A^3_\beta(p_0 - p_1) = \begin{pmatrix}-x_1-\frac12x_3-\frac12x_5\\-\frac{\sqrt3}2x_3-\frac1{2\sqrt3}x_5\\-\sqrt{\frac23}\,x_5\end{pmatrix}
	\end{align*}
are linearly independent and span the lattice~$\Lambda$. 
To verify	\eqref{eq:srsvol}, calculate
\[\det(g_1,g_2,g_3) = \frac{-1}{2\sqrt2}\,\det\begin{pmatrix}2x_1+x_2+x_6&-x_2+x_3&-2x_1-x_3-x_5\\-x_2-\frac13x_6&x_2+x_3+\frac23x_4&-x_3-\frac13x_5\\x_6&x_4&x_5\end{pmatrix}.\]
\vspace{-8mm}

\end{proof}

\begin{theorem}\label{thm:srs}
  Let $N$ be a triply periodic Steiner network in $\R^3$ with 
  quotient~$K_4$.  Then
  \begin{align*} 
	\frac{L^3}{V}\geq\frac{27}{\sqrt2}\approx19.09\,,
  \end{align*}
  where equality holds if and only if all edge lengths of $N$ coincide 
  and the lattice is body-centered cubic.
\end{theorem}
\begin{proof}
  We follow the strategy of the proof of Theorem~\ref{thm:ths}.
  For the present case, $L$ and $V$ are functions 
  of the six edge lengths, see Lemma~\ref{lem:srsformula}.

  We first verify the strict inequality $L^3>(27/\sqrt2)V$ 
  along the boundary of $[0,\infty)^6$ without the point~$0$.
  Assume that at least one $x_i$ vanishes. 
  By symmetry of $V$ and $L$ in all variables we may assume $x_6=0$. 
  Then the volume $V$ becomes
  \[V = \frac1{\sqrt2}(x_1x_2x_4+x_1x_2x_5+x_1x_3x_4+x_1x_3x_5+x_1x_4x_5+x_2x_3x_4+x_2x_3x_5+x_2x_4x_5)\,.\]
  This expression matches the volume~\eqref{eq:thsvol} 
  of a \ths~network with $x_6=0$ and $\alpha = \arccos(1/3)=\arcsin(2\sqrt2/3)$,
% note: 3/4 sin\alpha != 1/\sqrt 2 <=> \sin\alpha = 4/(3\sqrt 2) = 2\sqrt 2 / 3
% Thus arcsin(.) verifies our claim.  
% Then \cos\alpha = 1-(2\sqrt 2 /3)^2 = 1 - 8/9 = 1/9 ==> \alpha=arccos 1/3;
% we could equally well start from here, since this is the tetrahedral angle
% which srs has between two tangent planes.
  after exchanging~$x_3$ and~$x_5$.
% Suppose the above terms are labelled 1 to 8. 
% then they appear in \eq:thsvol in the order  2, 1, 5, 8, 4, 7, 3, 6 -- ok!
  Using \eqref{eq:thsinequality}, this proves, as desired
  \[L^3 \geq \frac{81}4V > \frac{27}{\sqrt2}V\,.\]

  Thus it suffices to minimize $L^3/V$ over the set 
  where all coordinates are strictly positive. 
  We maximize $V$ under the constraint $L=1$. 
  A critical point $z=(x_1,\ldots,x_6)\not=0$ satisfies 
  \[ \nabla V(z)= \lambda\nabla L(z) \,,\]
  where $\lambda\in\R$ is a Lagrange multiplier. 
  By \eqref{eq:srsvol} this is equivalent to
  \begin{align}\label{eq:srslagrange}
  \frac1{\sqrt2}\begin{pmatrix}x_2x_4+x_3x_4+x_2x_5+x_3x_5+x_4x_5+x_2x_6+x_3x_6+x_4x_6\\x_1x_4+x_3x_4+x_1x_5+x_3x_5+x_4x_5+x_1x_6+x_3x_6+x_5x_6\\x_1x_4+x_2x_4+x_1x_5+x_2x_5+x_1x_6+x_2x_6+x_4x_6+x_5x_6\\x_1x_2+x_1x_3+x_2x_3+x_1x_5+x_2x_5+x_1x_6+x_3x_6+x_5x_6\\x_1x_2+x_1x_3+x_2x_3+x_1x_4+x_2x_4+x_2x_6+x_3x_6+x_4x_6\\x_1x_2+x_1x_3+x_2x_3+x_1x_4+x_3x_4+x_2x_5+x_3x_5+x_4x_5\end{pmatrix}
  =\lambda\begin{pmatrix}1\\1\\1\\1\\1\\1\end{pmatrix}\,.
	\end{align}
  We claim this implies $z$ satisfies $x_1 = \ldots = x_6 = \frac16$.  
  Again we compute dot products of vectors with~\eqref{eq:srslagrange}.  
  For $(1,0,-1,1,0,-1)$ we obtain $-2x_1x_4+2x_3x_6=0$, 
  and $(0,1,-1,0,1,-1)$ gives $-2x_2x_5+2x_3x_6=0$.  
  Equivalently, 
  $x_1x_4 = x_2x_5 = x_3x_6$ or, since none of the coordinates can vanish,
  $x_4 = x_3x_6/x_1$ and $x_5 = x_3x_6/x_2$.  Using this, we conclude
  \begin{align*}
    0 = \left\langle\nabla V(z),(1,-1,0,0,0,0)\right\rangle 
     &= (x_2 - x_1)\frac{x_6(x_1x_2+x_1x_3+x_2x_3+x_3x_6)}{x_1x_2}\,, \\
	  0 = \left\langle\nabla  V(z),(0,0,1,-1,0,0)\right\rangle 
     &= (x_6 - x_1)\frac{x_1x_2+x_1x_3+x_2x_3+x_3x_6}{x_1}\,,\\
	  0 = \left\langle\nabla V(z),(0,0,0,0,1,-1)\right\rangle 
     &= (x_2 - x_3)\frac{x_6(x_1x_2+x_1x_3+x_2x_3+x_3x_6)}{x_1x_2}\,.
% check, for instance, the 3rd equ'n:
% 0 = (x_2-x_3)x_4 + (x_2+x_3+x_4)(x_6-x_5) 
%   = (x_2-x_3)x_6x_2x_3/(x_1x_2)+(x_1x_2+x_1x_3+x_3x_6)(x_2-x_3)x_6/(x_1x_2)
%   = (x_2-x_3)x_6[x_2x_3+x_1x_2+x_1x_3+x_3x_6]/(x_1x_2)  ok
	\end{align*}
  Therefore $x_1=x_2=x_3=x_6$ and, using $x_1x_4=x_2x_5=x_3x_6$,
  these must agree with $x_4=x_5$.  This proves the claim.
  Reasoning literally as in the proof of Theorem~\ref{thm:ths}
  concludes the proof.
\end{proof}
\begin{remark*}
  The proof of Theorem~\ref{thm:srs} asserts that if $x_6=0$
  the length and volume of the \srs\ network and the \ths\ network 
  with $\alpha=\arccos(1/3)$ agree.
  In particular, the combinatorial graphs of the
  networks agree, see Figure~\ref{fig:graphcontraction}.
\end{remark*}
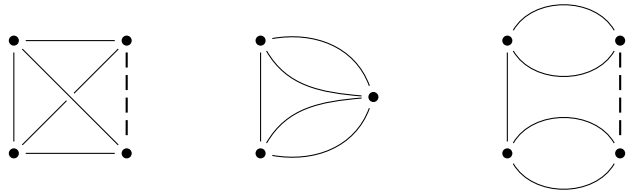
\begin{figure}
	\begin{subfigure}[c]{.25\linewidth}\centering
		\begin{tikzpicture}[scale=1.5]
			\coordinate[dot1](p1) at (0,0);
			\coordinate[dot1](p2) at (0,1);
			\coordinate[dot1](p3) at (1,1);
			\coordinate[dot1](p4) at (1,0);
			\draw(p1)--(p2)--(p3);
			\draw(p4)--(p1)--(p3);
			\draw[steiner2](p3)--(p4);
			\draw[cross](p2)--(p4);
		\end{tikzpicture}
	\end{subfigure}
	\begin{subfigure}[c]{.25\linewidth}\centering
		\begin{tikzpicture}[scale=1.5]
			\coordinate[dot1](p1) at (0,0);
			\coordinate[dot1](p2) at (0,1);
			\coordinate[dot1](p3) at (1,.5);
			\draw(p1)--(p2);
			\draw(p1) to[out=-10,in=-110](p3);
			\draw(p1) to[out=60,in=185](p3);
			\draw(p2) to[out=10,in=110](p3);
			\draw(p2) to[out=-60,in=175](p3);
		\end{tikzpicture}
	\end{subfigure}
	\begin{subfigure}[c]{.25\linewidth}\centering
		\begin{tikzpicture}[scale=1.5]
			\coordinate[dot1](p1) at (0,0);
			\coordinate[dot1](p2) at (0,1);
			\coordinate[dot1](p3) at (1,1);
			\coordinate[dot1](p4) at (1,0);
			\draw(p1)--(p2);
			\draw[steiner2](p3)--(p4);
			\draw(p1) to[out=-60,in=-120](p4);
			\draw(p1) to[out=60,in=120](p4);
			\draw(p2) to[out=-60,in=-120](p3);
			\draw(p2) to[out=60,in=120](p3);
		\end{tikzpicture}     
	\end{subfigure}
	\caption{The graph shown in the middle arises as a limit 
  of the graph $K_4$ (left) or of $D_1\gsquare D_2$ (right) 
  when the dashed edge is contracted.
	\label{fig:graphcontraction}}
\end{figure}

We would like to draw another consequence of Lemma~\ref{lem:srsangle}.
\begin{proposition}
  For each choice of edge lengths $x_1,\ldots,x_6>0$ 
  there exists a Steiner network in $\R^3$ with quotient $K_4$.
  Up to isometry, its vertices $p_1,\ldots,p_6$ 
  are uniquely given by~\eqref{eq:srscoords} as well as
  \begin{align*}
    p_4 &= p_2 + x_4\begin{pmatrix}0\\\frac1{\sqrt3}\\-\sqrt{\frac23}\end{pmatrix}\,, 
    & p_5 &= p_3 + x_5\begin{pmatrix}-\frac12\\-\frac1{2\sqrt3}\\-\sqrt{\frac23}\end{pmatrix}\,, 
    & p_6 &= p_1 + x_6\begin{pmatrix}\frac12\\-\frac1{2\sqrt3}\\-\sqrt{\frac23}\end{pmatrix}\,,
\end{align*}
  and the lattice is 
  $\Lambda = (p_6 - p_2)\mathbb Z + (p_4-p_3)\mathbb Z + (p_5-p_1)\mathbb Z$.
\end{proposition}
\begin{figure}[b]
	\centering
	\includegraphics[width=.412\linewidth]{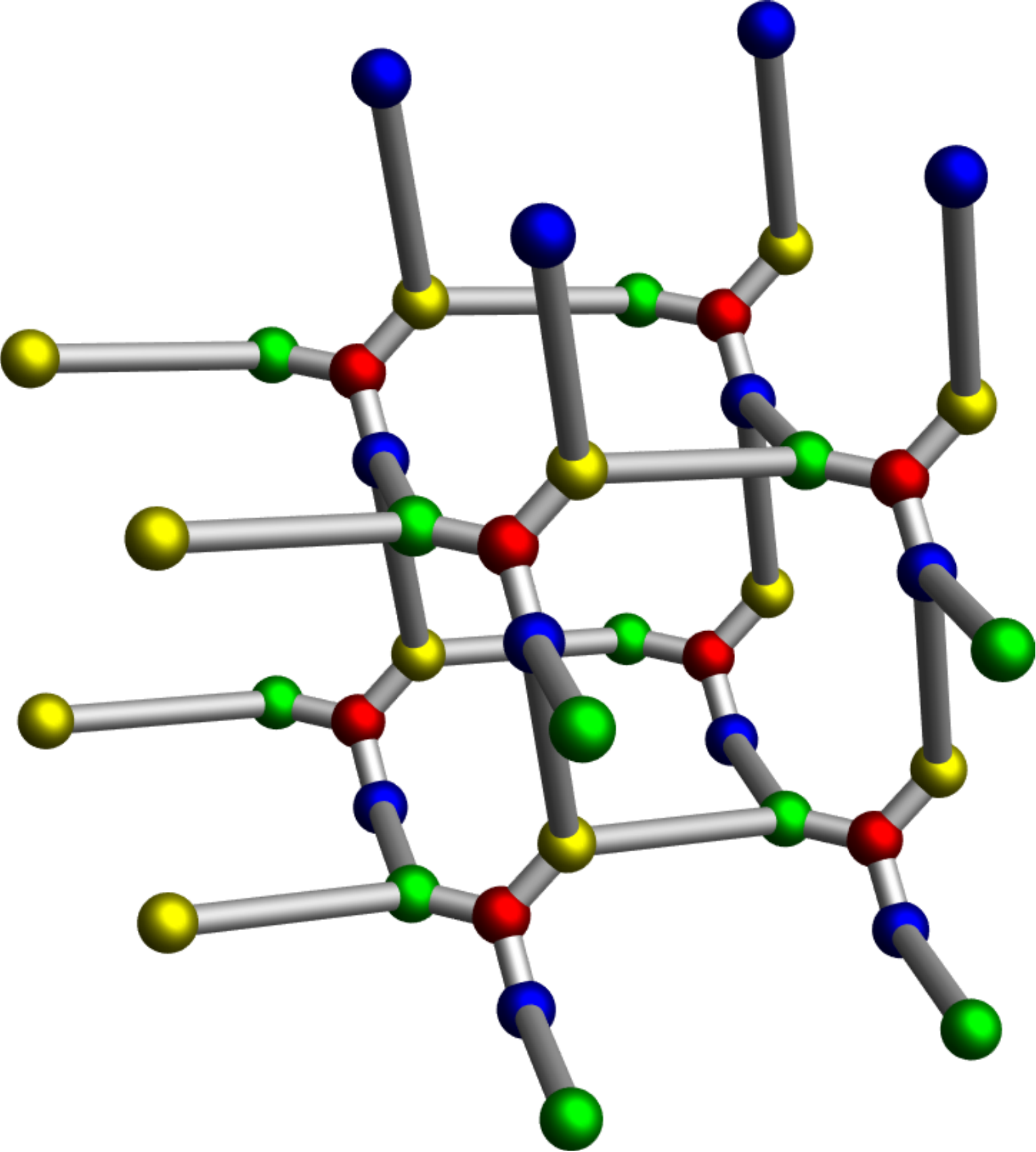}
	\caption{A triply periodic Steiner network with quotient $K_4$ where the lattice is primitive.  The eight vertices shown in red correspond to the eight vertices of a cube.}
	\label{fig:srscube}
\end{figure}
\noindent
% The vertices  $p_0+\Lambda,\ldots,p_3+\Lambda$ 
% determine edges labelled as in~Figure~\ref{fig:srs}: 
% \begin{align*}
%   e_1 &= \overline{p_0p_1} + \Lambda\,, 
%   & e_2 &= \overline{p_0p_2}+\Lambda\,, 
%   & e_3 &= \overline{p_0p_3} + \Lambda\,, \\
%   e_4 &= \overline{p_2p_4} + \Lambda\,, 
%   & e_5 &= \overline{p_3p_5}+\Lambda\,, 
%   & e_6 &= \overline{p_1p_6} + \Lambda\,.
% \end{align*}
Setting, for instance, $3x_1=3x_2=3x_3=x_4=x_5=x_6=3$ gives
\[
  g_1 = \begin{pmatrix}3\\-\sqrt3\\-\sqrt6\end{pmatrix}\,,\quad 
  g_2 = \begin{pmatrix}0\\2\sqrt3\\-\sqrt6\end{pmatrix}\,,\quad 
  g_3 = \begin{pmatrix}-3\\-\sqrt3\\-\sqrt6\end{pmatrix}\,.
\]
% |g_1|^2 = 9+3+6  =  |g_2|^2 = 12+6  =  |g_3|^2 = 9+3+6 , and
% scalar products of distinct vectors vanish.
These vectors have the same length and are orthogonal
so that the lattice is primitive cubic.
Moreover, $L^3/V = 16\sqrt2\approx22.63$. See Figure~\ref{fig:srscube}.

%%%%%%%%%%%%%%%%%%%%%%%%%%%%%%%%%%%%%%%%%%%%%%%%%%%%%%%%%%%%%%%%%%%%%%
\section{Homotopy from a minimizing \ths-network to a \texorpdfstring{$K_4$}{K4}-network which decreases length}\label{sec:homotopy}

We know from Theorem~\ref{thm:ths} that minimizing networks 
with quotient $D_1\gsquare D_2$ 
are part of the one-parameter family~\eqref{eq:thsequality} 
with a fixed lattice.
In the present section we show there is a continuous $1$-parameter family 
leading from a given minimizing \ths\ network to a network with 
smaller length and quotient~$K_4$.

The transition between the two distinct combinatorial types
is achieved via a network which has one degree-$4$ vertex
in the quotient.
There are two ways to split this vertex into two degree-$3$ vertices,
as is well-known from the Steiner tree problem on four vertices. 
See Figures~\ref{fig:graphcontraction} and~\ref{fig:homotopy} for 
the combinatorial picture, and~Figure~\ref{fig:thsk4} for the geometry.

\begin{theorem}\label{th:homotopy}
Let $N_{-1}$ be a minimizing \ths~network as in \eqref{eq:thsequality},
scaled such that $x_1=1$, and with lattice~$\Lambda_0$.  
Then there is a continuous family of triply periodic networks $N_t\subset\R^3$ for $t\in[-1,1]$ from $N_{-1}$ to a Steiner network $N_1$ 
with the following properties:
\begin{itemize}
	\item All $N_t$ have the same lattice $\Lambda_0$ and so the same volume~$V$.
	\item $N_t$ is a network with quotient graph $D_1\gsquare D_2$ for $-1\leq t<0$, and~$K_4$ for~$0<t\leq1$.
	\item The length $t\mapsto L(N_t)$ is non-increasing and $L(N_1)<L(N_{-1})$.
\end{itemize}
\end{theorem}
\noindent
We will specify the networks $N_t$ in terms of six generating vertices
$p_i^t$ for $i=1,\ldots,6$, as well as the straight segments joining pairs 
of these vertices given by Figure~\ref{fig:homotopy}. 
The lattice $\Lambda_0$ then generates~$N_t$.

Let us first describe the networks $N_t$ for negative~$t$
in terms of the \ths\ family~\eqref{eq:thsequality}. 
The given \ths~network~$N_{-1}$,
subject to \eqref{eq:thsequality} with $x_1=1$, has length $L(N_{-1})=9/2$ and
according to \eqref{eq:thslattice} its lattice $\Lambda_0$ is generated by
\[g_1 = \bigl(0,\sqrt3,0\bigr)\,,\quad g_2 = \bigl(0,0,-\sqrt3\bigr)\,,\quad 
  g_3 = \Bigl(\frac32,\frac{\sqrt3}2,\frac{\sqrt3}2\Bigr)\,.\]
The network $N_{-1}$ is uniquely determined by the edge length 
$x_5=:\xi\in (0,\frac12)$.  
For $-1\leq t<0$ we define $N_t$ as the \ths\ network,
subject to \eqref{eq:thsequality}, 
still with $x_1=1$ but with $x_5:=\vert t\vert\xi$.
This prescribes the six vertices $p_i^t$ for $t\in[-1,0)$
by \eqref{eq:thspoints0124} %\ref{eq:thspoints35}
to~\eqref{eq:thspoints6}.

\begin{figure}
	\centering
	\begin{subfigure}[t]{.3\linewidth}
		\centering
		\begin{tikzpicture}[x=1.5cm,y=1.5cm]
			\coordinate[dot2](p0) at (0,0);
			\coordinate[dot1](p1) at (.5,.866);
			\coordinate[dot3](p2) at (1,.866);
			\coordinate[dot4](p3) at (1.5,1.732);
			\coordinate[dot2](p4) at (0,1.732);
			\coordinate[dot4](p5) at (1.5,0);
			\coordinate[dot2](p6) at (2,1.732);
			\draw[steiner1](p0)--(p1);
			\draw[steiner2](p1)--(p4);
			\draw[steiner3](p1)--node[above]{$x_5$}(p2);
			\draw[steiner5](p2)--(p5);
			\draw[steiner4](p2)--(p3);
			\draw[steiner6](p3)--node[below]{$x_6$}(p6);
			\node[left] at (p1) {$p_1^{-1}$};
			\node[right] at (p2) {$p_2^{-1}$};
			\node[above] at (p6) {$p_6^{-1}$};
			\node[above] at (p4) {$p_4^{-1}$};
			\node[below] at (p0) {$p_0^{-1}$};
			\node[below] at (p5) {$p_5^{-1}$};
			\node[above] at (p3) {$p_3^{-1}$};
		\end{tikzpicture}
		\caption*{$t=-1$}
	\end{subfigure}
	\begin{subfigure}[t]{.3\linewidth}
		\centering
		\begin{tikzpicture}[x=1.5cm,y=1.5cm]
			\coordinate[dot2](p0) at (0,0);
			\coordinate[dot1](p1) at (.5,.866);
			\coordinate[dot3](p2) at (p1);
			\coordinate[dot4](p3) at (1,1.732);
			\coordinate[dot2](p4) at (0,1.732);
			\coordinate[dot4](p5) at (1,0);
			\coordinate[dot2](p6) at (2,1.732);
			\draw[steiner1](p0)--(p1);
			\draw[steiner2](p1)--(p4);
			\draw[steiner3](p1)--(p2);
			\draw[steiner5](p2)--(p5);
			\draw[steiner4](p2)--(p3);
			\draw[steiner6](p3)--(p6);
			\node[right] at (p1) {$p_1^0=p_2^0$};
			\node[above] at (p6) {$p_6^0$};
			\node[above] at (p4) {$p_4^0$};
			\node[below] at (p0) {$p_0^0$};
			\node[below] at (p5) {$p_5^0$};
			\node[above] at (p3) {$p_3^0$};
		\end{tikzpicture}
		\caption*{$t=0$}
	\end{subfigure}
	\begin{subfigure}[t]{.3\linewidth}
		\centering
		\begin{tikzpicture}[x=1.5cm,y=1.5cm]
			\coordinate[dot2](p0) at (0,0);
			\coordinate[dot3](p1) at (.75,.6);
			\coordinate[dot1](p2) at (.75,1);
			\coordinate[dot4](p3) at (1.5,1.732);
			\coordinate[dot2](p4) at (0,1.732);
			\coordinate[dot4](p5) at (1.5,0);
			\coordinate[dot2](p6) at (2,1.732);
			\draw[steiner1](p0)--(p1);
			\draw[steiner3](p2)--(p1);
			\draw[steiner2](p4)--(p2);
			\draw[steiner4](p3)--(p2);
			\draw[steiner5](p5)--(p1);
			\draw[steiner6](p6)--(p3);
			\node[below] at (p0) {$p_0^1$};
			\node[below] at (p1) {$p_1^1$};
			\node[above] at (p2) {$p_2^1$};
			\node[above] at (p3) {$p_3^1$};
			\node[above] at (p4) {$p_4^1$};
			\node[below] at (p5) {$p_5^1$};
			\node[above] at (p6) {$p_6^1$};
		\end{tikzpicture}
		\caption*{$t=1$}
	\end{subfigure}
	\caption{Homotopy of Theorem~\ref{th:homotopy}, schematically.  
The transition from the graph $D_1\gsquare D_2$ (left) to the $K_4$ graph (right) is via the graph (center) with a vertex of degree~$4$.}
	\label{fig:homotopy}
\end{figure}
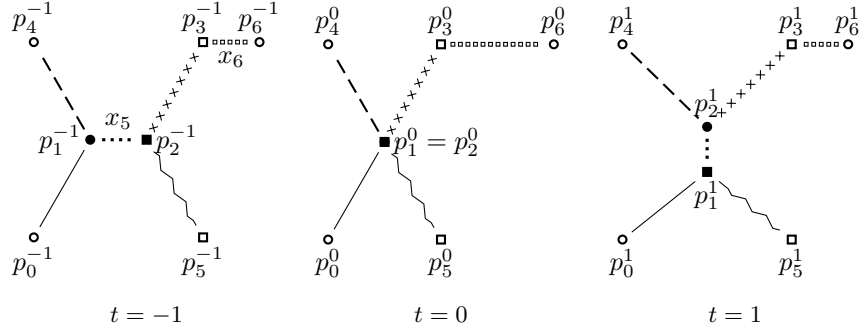

For $t=0$ the limiting data 
\[
  x_1=x_2=x_3=x_4=1\,,\quad x_5=0\,,\quad x_6=\frac12\,,\quad \alpha=\frac\pi2
\]
similarly defines a network~$N_0$ with $p^0_1=p^0_2=0$.
%
% The network $N_0$ has the vertices
% \begin{align*}
%	p_0^0 &= \bigl(-\frac12,-\frac{\sqrt3}2,0\bigr)\,,&
%	p_1^0 = p_2^0 &= (0,0,0)\,,&
%	p_3^0 &= \bigl(\frac12,0,\frac{\sqrt3}2\bigr) \,, 
% \end{align*}
% as well as
% \begin{align*}
%  p_4^0 &= p_0^0 + g_1\,, & p_5^0 &= p_3^0 + g_2\,, & p_6^0 &= p_0^0 + g_3\,.
% \end{align*}
Inspection of Figure~\ref{fig:homotopy} shows that under this condition
the network $N_0$ can also be understood as a 
limit of networks with quotient~$K_4$, where again the edge between the
points $p_1^t$ and~$p_2^t$ has length tending to~$0$ as $t\searrow0$.

To complete the proof of Theorem~\ref{th:homotopy}, 
let us now make the deformation of the network~$N_0$  
into a Steiner network with quotient~$K_4$ explicit.
\begin{lemma}
  There is a continuous family $N_t$, $0\leq t\leq1$, of networks 
  with lattice $\Lambda_0$, such that $N_1$ is a Steiner network,
  the length $t\mapsto L(N_t)$ is (strictly) decreasing, and for $0<t\leq1$
  the quotient graph is $K_4$.
\end{lemma}
\begin{proof}
  The network $N_1$ has the four vertices
  \begin{align*}
    p_0^1 &:= \Bigl(-\frac12,-\frac{\sqrt3}2,0\Bigr)\,, &
    p_1^1 &:= \Bigl(\frac14(\sqrt3-2),
                   \frac{\sqrt3}8(\sqrt2-2),\frac{\sqrt3}8(\sqrt2-2)\Bigr)\,,\\
    p_3^1 &:= \Bigl(\frac12(\sqrt3-1),0,\frac{\sqrt3}2\Bigr)\,,&
    p_2^1 &:= \Bigl(\frac14(\sqrt3-2),
                  -\frac{\sqrt3}8(\sqrt2-2),-\frac{\sqrt3}8(\sqrt2-2)\Bigr)\,,
  \end{align*}
  as well as the copies under the lattice
  \begin{align*}
    p_4^1 &= p_0^1 + g_1\,, & p_5^1 &= p_3^1 + g_2\,, & p_6^1 &= p_0^1 + g_3\,.
  \end{align*}
  We connect them with the six straight segments of Figure~\ref{fig:homotopy}.
  The $\Lambda_0$-orbit then defines the network~$N_1$, with quotient~$K_4$. 
  It can be checked by calculation that the balancing 
  equation~\eqref{eq:equilibrium} holds at each of the vertices 
  $p_0^1,\ldots,p_3^1$, and so $N_1$ is a Steiner network. 
  The length of $N_1$ is 
  \begin{align*}
    L(N_1) 
    &= \vert p_1^1 - p_0^1\vert + \vert p_2^1 - p_1^1\vert 
    + \vert p_2^1 - p_4^1\vert + \vert p_1^1 - p_5^1\vert 
    + \vert p_2^1 - p_3^1\vert + \vert p_6^1-p_3^1\vert\\
    &= \frac{\sqrt3}2+\frac{\sqrt3}2(\sqrt2-1)+\frac{\sqrt3}2
     +\frac{\sqrt3}2+\frac{\sqrt3}2+\frac{\sqrt3}2(\sqrt3-1)\\
    &= \frac{\sqrt3}2(2+\sqrt2+\sqrt3). % \equiv 4.4567957
  \end{align*}

  We define $N_t$ for $t\in(0,1)$ as a convex combination of the vertices 
  of $N_0$ and $N_1$, 
  \[
    p_i^t := (1-t)p_i^0+tp_i^1\,, \qquad i=1,\ldots, 6\,,
  \]
  again connected with six straight segments as in Figure~\ref{fig:homotopy}.

  That $L(N_t)$ is a decreasing function of $t\in[0,1]$
  follows from three facts.  
  First, $L(N_1)$ is strictly less than $L(N_{0})=9/2$.  
  Second, the function $t\mapsto L(N_t)$ is critical at $t=1$
  since $N_1$ is Steiner.
  Third, each term $t\mapsto\vert p_i^t-p_j^t\vert$ of $L(N_t)$ is
  convex on $[0,1]$, and so is $t\mapsto L(N_t)$.
\end{proof}

\nocite{*}
\bibliographystyle{amsplain}
\bibliography{steiner}
% Journal names according to http://www.ams.org/msnhtml/serials.pdf

\end{document}